\documentclass[12pt,a4paper,oneside,final,reqno]{amsart}

\usepackage[backend=bibtex, style=numeric, maxnames=50]{biblatex}
\renewbibmacro{in:}{}

\allowdisplaybreaks

\usepackage[T1]{fontenc}
\usepackage[utf8]{inputenc}
\usepackage[dvipsnames]{xcolor}
\usepackage[english]{babel}
\usepackage{mathtools, amsthm, amssymb, amscd} 
\usepackage{newtxtext, newtxmath} 
\usepackage[DIV=11]{typearea} 
\usepackage{hyperref} 
\usepackage[notref,notcite]{showkeys} 
\usepackage{todonotes}
\usepackage{enumitem} 
\usepackage{xypic} 
\usepackage{tikz}
\usetikzlibrary{cd, decorations.markings, arrows} 

\usepackage[capitalise]{cleveref}

\DeclareMathAlphabet{\mathcal}{OMS}{cmsy}{m}{n} 

\definecolor{DarkPurple}{rgb}{0.40,0.0,0.20}
\hypersetup{
	colorlinks =true,
	linkcolor= DarkPurple, 
	citecolor = NavyBlue 
} 

\usetikzlibrary{cd, decorations.markings, arrows, matrix, calc}

\newcommand{\G}{\mathcal{G}}
\newcommand{\E}{\mathcal{E}}
\newcommand{\A}{\mathcal{A}}

\newcommand{\Fp}{F_{\lambda}^p}
\newcommand{\Fq}{F_{\lambda}^q}

\newcommand{\T}{\mathbb{T}}
\newcommand{\C}{\mathbb{C}}
\newcommand{\R}{\mathbb{R}}
\newcommand{\Z}{\mathbb{Z}}
\newcommand{\N}{\mathbb{N}}

\newcommand{\Cred}{C_{r}^{*}}

\DeclareMathOperator{\supp}{supp}

\newtheorem{lemma}{Lemma}[section]
\newtheorem{corollary}[lemma]{Corollary}
\newtheorem{theorem}[lemma]{Theorem}
\newtheorem*{theorem*}{Theorem}
\newtheorem{proposition}[lemma]{Proposition}
\newtheorem{introtheorem}{Theorem}

\theoremstyle{definition}
\newtheorem{definition}[lemma]{Definition}
\newtheorem{example}[lemma]{Example}
\newtheorem{remark}[lemma]{Remark}
\newtheorem{question}[lemma]{Question}

\title[]{Polynomial growth and property $RD_p$ for étale groupoids with applications to $K$-theory}
\date{}

\author[1]{Are Austad, Eduard Ortega, Mathias Palmstrøm}

\address{Department of Mathematics and Computer Science, Faculty of Science, SDU Odense -- University of Southern Denmark, Odense, Denmark}
\email{are@sdu.dk}

\address{Department of Mathematical Sciences, Faculty of Information Technology and Electrical Engineering, NTNU -- Norwegian University of Science and Technology, Trondheim, Norway}
\email{eduard.ortega@ntnu.no}

\address{Department of Mathematical Sciences, Faculty of Information Technology and Electrical Engineering, NTNU -- Norwegian University of Science and Technology, Trondheim, Norway}
\email{mathias.palmstrom@ntnu.no}


\numberwithin{equation}{section} 

\begin{document}
	\renewcommand{\thefootnote}{}
	
	\begin{abstract}
		We investigate property $RD_p$ for étale groupoids and apply it to the computation of the $K$-theory of reduced groupoid $L^p$-operator algebras. In particular, under the assumption of polynomial growth, we show that the $K$-theory groups for a reduced groupoid $L^p$-operator algebra are independent of $p\in (1, \infty)$. We apply the results to coarse groupoids and graph groupoids. 
	\end{abstract}
\maketitle

\section{Introduction}
Computing $K$-theory is a central problem in the study of operator algebras. For $C^*$-algebras one of the most powerful tools is the Baum-Connes assembly map  
\begin{equation*}
	\mu_2:K^{\mathrm{top}}(\Gamma)\to K_*(C^*_r(\Gamma)) ,
\end{equation*}
where $K^{\mathrm{top}}(\Gamma)$ is the $K$-theory of the classifying space for proper actions of $\Gamma$ and $K_*(C^*_r(\Gamma))$ is the $K$-theory of the reduced group $C^*$-algebra. Groups for which the map $\mu_2$ is an isomorphism are said to satisfy the Baum-Connes conjecture,  
of which groups with the Haagerup property and Gromov hyperbolic groups are among the most important examples. See \cite{SurveyBC} and references therein for a nice survey on the Baum-Connes conjecture. The assembly map has been extended from groups to actions of groups on $C^*$-algebras (Baum-Connes assembly map with coefficients) \cite{BCH}, to groupoids \cite{TuBaumConnes}
and for metric embeddings into Hilbert spaces \cite{Yu00}. 
But while the Baum-Connes conjecture for groups is still open, there are known counterexamples in the groupoid case \cite{HLS}. 
For Banach algebras there exists also a Baum-Connes assembly map due to Lafforgue 
\begin{equation*}
	\mu_\mathcal{A}:K^{\mathrm{top}}(\Gamma)\to K_*(\mathcal{A}(\Gamma)) ,
\end{equation*}
where $\mathcal{A}(\Gamma)$ is any Banach algebra that is an unconditional completion of $C_c(\Gamma)$, for example $L^1(\Gamma)$ \cite{Lafforgue1}. This map was shown  
to be an isomorphism when the group $\Gamma$ belongs to a large class $\mathcal{C}'$, called the Lafforgue class, which 
includes  
hyperbolic groups and semi-simple real Lie groups. The assembly map for unconditional completions of groupoids has been also studied by Paravicini \cite{Paravicini}. 
There is then the question of when the $K$-groups of $C^*_r(\Gamma)$ and $\mathcal{A}(\Gamma)$ are isomorphic. Lafforgue proved in \cite{Lafforgue2} that for discrete groups satisfying the property of rapid decay, which we will denote by property $RD$,  
there exists a unconditional completion of $C_c(\Gamma)$ that is contained in $C^*_r(\Gamma)$  and stable under holomorphic functional calculus. Then, the result of Connes \cite[Appendix C, Proposition 3]{NoncommutativeGeomConnes} shows that their $K$-theories are isomorphic. Property $RD$ for discrete groups was first established for free groups by Haagerup in \cite{HaagerupAnExample...} and properly introduced and studied by Jolissaint in \cite{JolissaintRD}, who verified it for groups with polynomial growth and hyperbolic groups. See \cite{SurveyOnRDChatterji} for a nice exposition about the RD property for discrete groups. The rapid decay property for actions of locally compact groups acting on $L^p$-spaces was introduced by Liao and Yu in \cite{kTheoryOfBAandRD} and for étale groupoids by Hou in \cite{hou2017spectral}. Properties similar to RD has also been studied in the case of reduced crossed products of discrete groups  
by Christensen in \cite{christensen2021DynamicalRapidDecay}, Chen and Wei in \cite{ChenWeiSRDandCrossedProducts}, and Ji and Schweitzer in \cite{JiSchweitzerSpectralInvarianceCrossedProducts}. 

Recently, the class of Banach algebras of $L^p$-operator algebras has attracted a lot of attention \cite{CGTRigidityResultsForLpOp,Chung&Li,GardellaModernLookofLp,KTheoryCuntzPhillips,Zhang&ZhouLp}. These are Banach algebras which admit an isometric representation on an $L^p$-space, and therefore generalize $C^*$-algebras to non-selfadjoint closed subalgebras of bounded operators on $L^p$-spaces. 
Analogously to $C^*_r(\Gamma)$ one can define the reduced group $L^p$-operator algebra $F^p_\lambda(\Gamma)$ for a discrete group $\Gamma$, as the Banach subalgebra of $B(\ell^p(\Gamma))$  generated by the image of the left regular representation of $\Gamma$, so in particular $C^*_r(\Gamma)=F^2_\lambda(\Gamma)$.  

In order to compute the K-groups of $\Fp(\Gamma)$, in \cite{kTheoryOfBAandRD} it was announced that Kasparov and Yu had defined the $L^p$ version of the Baum-Connes assembly map, for $p\in [1,\infty)$, and it was proved that for groups in the Lafforgue class satisfying property $RD_p$, $K_*(F^p_\lambda(\Gamma))$ is independent of $p \in [1, \infty)$. 
In particular, they are all isomorphic to $K_*(C^*_r(\Gamma))$ and consequently isomorphic to $K^{\mathrm{top}}(\Gamma)$. The results of Liao and Yu hinted at the possibility that the $K$-theory of $L^p$-operator algebras is independent of $p$. This was previously pointed out by Phillips when he computed the $K$-theory of the $L^p$-analog of the Cuntz algebras and the UHF algebras \cite{KTheoryCuntzPhillips}. In this paper, we aim  
to prove an analogous result of Liao and Yu \cite[Theorem 1.5]{kTheoryOfBAandRD} but for 
the $L^p$-operator algebra $F^p_\lambda(\mathcal{G})$ associated to an étale groupoid $\mathcal{G}$.
That is, we will show that for certain classes of étale groupoids the groups $K_*(F^p_\lambda(\mathcal{G}))$ are independent of $p$. 

The added complexity of multiple units and general lack of symmetry in the groupoid setting forces us to attack the problem differently from Liao and Yu.
Our strategy to is to consider 
property $RD$ for étale groupoids, with respect a length function $l$, defined by Hou \cite[Definition 3.2]{hou2017spectral} and extend it to $L^p$-operator algebras by defining property $RD_p$.  
For $p,q\in (1,\infty)$ 
such that $\frac{1}{p}+\frac{1}{q}=1$, we show that whenever a groupoid $\G$  satisfies properties $RD_p$  and $RD_q$ we can construct a Fréchet subalgebra $S^l_p(\G)$ of both
$F^p_\lambda(\G)$ and $F^q_\lambda(\G)$ which is stable under holomorphic functional calculus. This allows us to prove our first main result.

\begin{introtheorem}[cf. \ref{thm: RD_q and RD_p implies isomorphisms in K-theory}]
	Let $\G$ be an étale groupoid endowed with a continuous length function for which it has property $RD_p$ and $RD_q$, where $p,q \in (1, \infty)$ are Hölder conjugate. Then   
	\begin{equation*}
		K_*(F^p_\lambda(\G))\cong K_*(S^l_p(\G)) \cong K_*(F^q_\lambda(\G)) .
	\end{equation*}
\end{introtheorem}

Of special interest when studying property $RD_p$ are groupoids of polynomial growth, which we treat in Section \ref{sec: Length Functions Rapid Decay and Polynomial Growth}. It turns out that groupoids of polynomial growth have property $RD_p$ for all $p \in [1,\infty)$, which leads us to our second main result of the paper. 

\begin{introtheorem}[cf. \ref{thm: Polynomial growth implies all K-groups are isomorphic}] 
	Let $\G$ be an étale groupoid endowed with a continuous length function for which it has polynomial growth. Then the groups $ K_{\ast}(F_{\lambda}^{p}(\G))$, for $p \in (1,\infty)$, are all isomorphic.
\end{introtheorem}	

Polynomial growth of an étale groupoid can sometimes be read off from an analogous property in the setting the groupoid arises. This is in particular the case for a coarse groupoid associated to a uniformly locally finite extended metric space $(X,d)$. In this case, polynomial growth of the extended metric space will imply polynomial growth of the coarse groupoid. Applying Theorem \ref{thm: Polynomial growth implies all K-groups are isomorphic} in this case, we prove 
the following.
\begin{introtheorem}[cf. \ref{corollary:roe-algebra-k-theory-results}]
	Let $(X,d)$ be a uniformly locally finite extended metric space. 
	Denote by $ \overline{B}(x,r)$ the closed $r$-ball with center $x$, and let $B^p_u (X,d)$ denote the uniform $L^p$-Roe algebra of $(X,d)$. If there is a polynomial $f$ such that $\vert \overline{B}(x,r)\vert \leq f(r)$ for all $x \in X$ and all $r \geq 0$, then the $K$-theory groups $K_* (B^p_u (X,d))$,  for $*=0,1$ are independent of $p \in (1,\infty)$.
\end{introtheorem}

Another relevant example of étale groupoids are those arising from directed graphs \cite{KPRR}. Graph groupoids produce nice classes of algebras which includes the Toeplitz algebra, Cuntz-Krieger algebras and the algebra of continuous functions on the quantum lens space \cite{HongSzy}. We give a combinatorial condition on the graph that exactly determines when the graph groupoid has polynomial growth (Proposition \ref{Growth_Graphs}). 

The paper is structured as follows. In Section \ref{sec: preliminaries}, we recall basic results on $K$-theory for Fréchet algebras, étale groupoids, and reduced groupoid $L^p$-operator algebras. We then define property $RD_p$ and polynomial growth for étale groupoids and derive some basic results and permanence properties concerning these in Section \ref{sec: Length Functions Rapid Decay and Polynomial Growth}. In Section \ref{sec: applications of polynomial growth and RDp}, we apply property $RD_p$ and polynomial growth to derive $K$-theoretic results. This is done by constructing spectral invariant dense Fréchet subalgebras of the reduced groupoid $L^p$-operator algebras. In Section \ref{sec:Examples}, we apply the results of Section \ref{sec: applications of polynomial growth and RDp} to specific examples. Of note are the reduced $L^p$-operator algebras associated to coarse groupoids of uniformly locally finite coarse spaces and to groupoids associated to finite directed graphs. Lastly, in Appendix \ref{appendix:twisted-case} we indicate how to extend  the main results to the setting of $2$-cocycle twisted reduced groupoid $L^p$-operator algebras.

\section{Preliminaries} \label{sec: preliminaries}

\subsection{$K$-theory for Fréchet Algebras}
A \emph{Fréchet space} is a locally convex space which satisfies the following three conditions:
\begin{itemize}
	\item it is a Hausdorff space;
	\item its topology is induced by a countable family of semi-norms $\{ \lVert \cdot \rVert_{k} \}_{k \in \N_0} $, where $\N_0 = \{0\} \cup \N$;
	\item it is complete with respect to the family of semi-norms $\{ \lVert \cdot \rVert_k \}_{k \in \N_0}$.
\end{itemize} 
In a Fréchet space, a sequence $\{ f_n \}_{n \in \N}$ converges to $f $ if and only if $f_n$ converges to $f$ with respect to each semi-norm $\lVert \cdot \rVert_k $, $k \in \N_0$. We can, and will, always assume that the countable family of semi-norms are increasing.  Moreover, we shall only consider Fréchet spaces arising from a countable collection of norms. 

By a \emph{Fréchet algebra}, we mean an associative $\C$-algebra which is also a Fréchet space and such that the multiplication is jointly continuous. Since we are assuming that the countable family of norms $\{ \lVert \cdot \rVert_n \}_{n \in \N_0}$ is an increasing family, the multiplication is jointly continuous if and only if for each $n \in \N_0$ there exist $C_n > 0$ and $m \geq n $, such that $\lVert ab \rVert_{n} \leq C_n \lVert a \rVert_m \lVert b \rVert_m $, for all $a,b \in \A $. In what follows, we shall define the $K$-groups associated to Fréchet algebras whose underlying Fréchet space arises from a countable family of norms. The definition we shall employ is Definition 7.1 from \cite{KTheoryFrechetPhillips}, wherein Phillips defines the K-groups for Fréchet algebras arising from a countable family of \emph{sub-multiplicative} semi-norms. However, \cite[Definition 7.1]{KTheoryFrechetPhillips} makes sense in our case as well. 

To be able to state this definition, we need some preparations. Let $\A$ be a Fréchet algebra whose Fréchet topology is induced by a countable family of norms $\{ \lVert \cdot \rVert_{k} \}_{k \in \N}$. If $\A$ is not unital, one may adjoin a unit in the same manner as one does for a Banach algebra to obtain a unital Fréchet algebra $\tilde{\A} = \C \oplus \A $, whose topology is induced by the countable family of norms $\{ \lVert \cdot \rVert_{k, \sim} \}_{n \in \N} $, where $\lVert (\mu,a) \rVert_{k,\sim} := | \mu | + \lVert a \rVert_k $, for $(\mu, a) \in \tilde{A}$. Also, for any $n \in \N$, the matrix algebra $M_n (\A)$ is naturally a Fréchet algebra via the countable family of norms $\{ \lVert \cdot \rVert_{k,n} \}_{k \in \N_0} $, given by $\lVert a \rVert_{k,n} := \sum_{i,j = 1}^{n} \lVert a_{i,j} \rVert_{k} $. Using the embeddings $a \mapsto \text{diag}(a,0)$, we define $M_{\infty}(\A)$ as the algebraic direct limit of the matrix algebras $M_n (\A)$.

\begin{definition} \label{def: definition of the K-groups for Fréchet algebras}
	Let $\A$ be a unital Fréchet algebra. We define $K_0 (\A)$ as the Grothendieck group of the semigroup of algebraic equivalence classes of idempotents in $M_{\infty} (\A)$. Using the embeddings $u \mapsto \text{diag}(u,1) $, we define $K_1 (\A) = \varinjlim GL_n (\A)/GL_n (\A)_0 $, where the $GL_n (\A)$ are the invertible matrices in the Fréchet algebra $M_n (\A)$ endowed with the induced topology, and $GL_n (\A)_0$ is the normal subgroup given by the path component of the identity. If $\A$ is not unital, $K_{\ast}(\A)$ is defined to be the kernel of the naturally induced map from $K_{\ast}(\tilde{\A})$ to $K_{\ast}(\C)$.
\end{definition}
When $\A$ is a Banach algebra, the above defines its usual $K$-groups. The following lemma is a special case of \cite[Lemma 1.2]{SchweitzerSpectralInvarianceAndHolomorphicFuncCalc}. Recall that a Fréchet subalgebra $\A$ of a Banach algebra $A$ is spectral invariant if the invertible elements of $\tilde{\A}$ are precisely those elements of $\tilde{\A}$ which are invertible in $\tilde{A}$.
\begin{lemma} \label{lem: specrtal invariant iff stable under holomorphic functional calc.}
	Suppose that $\A$ is a subalgebra of a Banach algebra $A$, and that under a possibly finer topology than the one inherited by $A$, $\A$ is a Fréchet algebra. Then $\A$ is closed under holomorphic functional calculus in $A$ if and only if $\A$ is spectral invariant in $A$.
\end{lemma}
The proof of \cite[Corollary 7.9]{KTheoryFrechetPhillips} together with Lemma \ref{lem: specrtal invariant iff stable under holomorphic functional calc.} and \cite[Chapter 3, Appendix C, Proposition 2 and 3]{NoncommutativeGeomConnes} give the next lemma.
\begin{lemma} \label{lem: Connes lemma that spectral invariance induce isomorphism in K-theory}
	Let $\A$ be a subalgebra of a Banach algebra $A$, and suppose $\A$ is a Fréchet algebra under a topology $\tau$ which is finer than the one inherited from $A$. If $\A$ is dense and spectral invariant in $A$, then the inclusion $(\A , \tau) \hookrightarrow A$ induces isomorphisms in $K$-theory $K_{\ast}(\A) \cong K_{\ast}(A)$, $\ast = 0,1$.
	\begin{proof}
		For $K_0 $, since $\A$ is a dense spectral invariant subalgebra of $A $, Lemma \ref{lem: specrtal invariant iff stable under holomorphic functional calc.} and \cite[Chapter 3, Appendix C, Proposition 3 a)]{NoncommutativeGeomConnes} imply together that $K_{0}(\A) \cong K_0 (A)$. 
		
		Lemma \ref{lem: specrtal invariant iff stable under holomorphic functional calc.} together with \cite[Chapter 3, Appendix C, Proposition 2]{NoncommutativeGeomConnes} gives that $M_n (\A)$ is spectral invariant in $M_n (A)$, for every $n \in \N $. Thus, $GL_n (\tilde{\A}) = GL_n (\tilde{A}) \cap M_n (\tilde{\A}) $. By $GL_n (\tilde{\A}) $, we shall mean the subgroup $GL_n (\tilde{\A}) \subset GL_n (\tilde{A})$ with the induced topology, whilst by $GL_n (\tilde{\A}, \tau)$, we shall mean the same group but with the topology induced from $M_n (\tilde{\A}, \tau) $. To see the isomorphism for $K_1 $, we shall show that $GL_n (\tilde{\A}, \tau)_0 = GL_n (\tilde{\A})_0 $, for each $n \in \N $. We shall argue similarly as in the proof of \cite[Corollary 7.9]{KTheoryFrechetPhillips}. Any path of invertibles in $M_n (\tilde{\A})$ that is continuous in the Fréchet topology is also a continuous path of invertibles in the induced topology. Thus, $GL_n (\tilde{\A}, \tau)_0 \subset GL_n (\tilde{\A})_0 $. Conversely, assume we are given a path of invertibles that is continuous in the induced topology. By compactness, we may cover the path by a finite number of balls whose radius is so small that the elements on any straight line segment joining two of its points are all in $GL_n (\tilde{A})$. Since the path is contained in $GL_n (\tilde{\A)} = M_n (\tilde{\A}) \cap GL_n (\tilde{A})$, we may find piece-wise linear paths gluing together to give a path in $GL_n (\tilde{\A})$, continuous in the Fréchet topology, with same end- and start point as the original path. This means that $GL_n (\tilde{\A})_0 \subset GL_n (\tilde{\A}, \tau)_0$, and so $GL_n (\tilde{\A}, \tau)_0 = GL_n (\tilde{\A})$. Thus, for each $n \in \N$, 
		\begin{equation*}
			GL_n (\tilde{\A})/GL_n (\tilde{\A})_0 = GL_n (\tilde{\A} , \tau)/GL_n (\tilde{\A}, \tau)_0 .	
		\end{equation*}
		This, together with \cite[Chapter 3, Appendix C, Proposition 3 b)]{NoncommutativeGeomConnes}, gives us that
		\begin{equation*}
			K_1 (\tilde{\A}) = \varinjlim GL_n (\tilde{\A} , \tau)/GL_n (\tilde{\A} , \tau)_0 = \varinjlim GL_n (\tilde{\A})/GL_n (\tilde{\A})_0 \cong K_1 (\tilde{A}) ,
		\end{equation*}
		and so by definition, $K_1 (\A) \cong K_1 (A)$.
	\end{proof}
\end{lemma}
For information regarding $K$-theory for Banach algebras, we refer the reader to Blackadars book \cite{blackadarKTheory}. See also \cite{KTheoryFrechetPhillips} for the so called representable $K$-theory defined for locally multiplicatively convex Fréchet algebras.

\subsection{Étale Groupoids} \label{subsec: preliminaries on étale groupoids}
A \emph{groupoid} is a set $\G$ equipped with a partially defined multiplication (composition) $\G^{(2)} \to \G  , (x,y) \mapsto xy$, where $\G^{(2)} \subset \G \times \G$ is the set of composable pairs, and with an inverse map $\G \to \G  , x \mapsto x^{-1}$, such that the following three axioms are satisfied:

\begin{itemize}
	\item[(1)]If $(x , y), (y, z) \in \G^{(2)}$, then $(x y , z) , (x , yz) \in \G^{(2)}$ and $(xy ) z = x (y z)$;
	\item[(2)]$(x^{-1})^{-1} = x$, for all $x \in \G$;
	\item[(3)] $(x, x^{-1}) \in \G^{(2)}$, for all $x \in \G$, and when $(x,y) \in \G^{(2)}$, we have $x^{-1} (xy) = y$ and $(xy)y^{-1} = x$.
\end{itemize}

The set $\G^{(0)} := \{ x x^{-1} \colon x \in \G \}$ is called the \emph{unit space} of the groupoid $\G$, and the maps $r \colon \G \to \G , r(x) = x x^{-1}$ and $s \colon \G \to \G , x \mapsto x^{-1} x$ are called the \emph{range} and \emph{source} maps respectively. We have that $(x,y) \in \G^{(2)}$ if and only if $s(x) = r(y)$. 

A groupoid endowed with a topology such that the operations of multiplication and inversion are continuous is called a \emph{topological groupoid}. Moreover, if the topology is locally compact Hausdorff such that the range map, and therefore also the source map, is a local homeomorphism, the groupoid is said to be \emph{étale}. The open sets $U \subset \G$ for which both $s |_U$ and $r |_U$ are homeomorphisms are called \emph{bisections}. Thus, an étale groupoid is one whose topology has a basis of open bisections. An étale groupoid is said to be \emph{ample} if its topology admits a basis of compact open bisections. 

For any $X \subset \G^{(0)} $, we denote by $\G_X = \{x \in \G \colon s(x) \in X\} $ and $\G^{X} = \{x \in \G \colon r(x) \in X\}$. We shall write $\G_u$ and $\G^{u}$ instead of $\G_{\{u\}}$ and $\G^{\{u\}}$ whenever $u \in \G^{(0)}$ is a unit. The set $\G |_X = \G_X \cap \G^{X} = \{x \in \G \colon r(x) , s(x) \in X\}$ is a subgroupoid of $\G$, with unit space $X$, called the \emph{restriction} of $\G$ to $X$. 
The \emph{isotropy group} at a unit $u \in \G^{(0)}$ is the group $\G_{u}^{u} := \G_u \cap \G^{u}$ and the \emph{isotropy bundle} is
\begin{equation*}
	\text{Iso}(\G) := \{ x \in \G \colon s(x) = r(x) \} = \bigsqcup_{u \in \G^{(0)}} \G_{u}^{u} .
\end{equation*}
A groupoid $\G$ is said to be \emph{principal} if its isotropy bundle coincides with its unit space and is called \emph{effective} if the interior of its isotropy bundle coincides with its unit space. 

The \emph{orbit} of a unit $u \in \G^{(0)}$ is the set $r(\G_u) = \{r(x) \colon s(x) = u\} = s(\G^{u})$. A subset $F \subset \G^{(0)}$ is said to be \emph{full} if it meets every orbit, that is, $F \cap r(\G_u) \neq \emptyset$, for each unit $u \in \G^{(0)}$. Given two sets $U,V \in \G$, we define $UV = \{xy \colon s(x) = r(y) \text{ and } x \in U , y \in V \}$ and $U^{-1} = \{x^{-1} \colon x \in U\}$. 

A \emph{groupoid homomorphism} between two étale groupoids $\G$ and $\mathcal{H}$, is a map $\phi \colon \G \to \mathcal{H}$ such that if $(x,y) \in \G^{(2)}$, then $(\phi(x) , \phi(y)) \in \mathcal{H}^{(2)}$, and in this case $\phi(xy) = \phi(x) \phi(y)$. Two étale groupoids are said to be isomorphic if there is a bijective groupoid homomorphism that is also a homeomorphism. 

We say that two étale groupoids $\G$ and $\mathcal{H}$ are \emph{Kakutani equivalent} if there are full clopen subsets $U \subset \G^{(0)}$ and $V \subset \mathcal{H}^{(0)}$ such that $\G |_{U} \cong \mathcal{H} |_{V}$, as étale groupoids. In the setting of ample groupoids with $\sigma$-compact unit spaces, this notion of equivalence of groupoids is equivalent to, among other things, Morita equivalence of groupoids and equivalence of groupoids in the sense of Renault (see \cite{AmpleGroupoidsAndHomology}). For more background material on étale groupoids, we refer the reader to the books of Jean Renault \cite{renault2006groupoid}, Alan Paterson \cite{patersonGroupoids} and Aidan Sims \cite{SimsNotesOnGroupoids}. 

\subsection{Reduced Groupoid $L^p$-Operator Algebras} \label{subsec: preliminaries on reduced Lp operator algebras}
For an étale groupoid $\G$ and $p \in [1, \infty)$, one can construct an $L^p $-operator algebra in a similar manner as to how one constructs its reduced $C^{\ast}$-algebra. We recall this construction in what follows. 

Since $\G$ is étale, the fibers $\G_u$, for $u \in \G^{(0)}$, are discrete. Let $C_c (\G)$ denote the space of continuous compactly supported functions on $\G$. We endow $C_c (\G)$ with the convolution product, which for $f,g \in C_c (\G)$ is given by 
\begin{equation*}
	f \ast g (x) = \sum_{y \in \G_{s(x)}} f(x y^{-1}) g(y) ,
\end{equation*}
for $x \in \G$. Under the convolution product, $C_c (\G)$ is an associative $\C$-algebra. Moreover, under the \emph{$I$-norm} given by 
\begin{equation*}
	\lVert f \rVert_I = \max \Big\lbrace \sup_{u \in \G^{(0)}} \sum_{x \in \G_u} |f(x)| , \sup_{u \in \G^{(0)}} \sum_{x \in \G^{u}} |f(x)| \Big\rbrace ,
\end{equation*}
$C_c (\G)$ becomes a normed algebra. Fix any unit $u \in \G^{(0)}$. The operator $\lambda_u (f)$ associated to $f \in C_c(\G)$, is the operator given by 
\begin{equation*}
	\lambda_u (f) (\xi) (x) = \sum_{y \in \G_u} f(x y^{-1}) \xi (y) ,
\end{equation*}
for $x \in \G_{u}$ and $\xi \in C_c (\G_u)$. The map $\lambda_u \colon C_c (\G) \to B(\ell^p (\G_u))$ is a contractive representation of $C_c(\G)$, and is called the \emph{left regular representation} at $u$. The \emph{reduced $L^p $-operator algebra} associated to $\G$ is denoted $F_{\lambda}^p (\G)$ and is the completion of $C_c(\G)$ under the norm $\lVert f \rVert_{\Fp(\G)} := \sup_{u \in \G^{(0)}} \lVert \lambda_{u}(f) \rVert$. Since $\bigoplus_{u \in \G^{(0)}} \lambda_u$ is an isometric representation of $F_{\lambda}^p (\G)$ on an $L^p $-space, $F_{\lambda}^p (\G)$ is indeed an $L^p $-operator algebra. $F_{\lambda}^p (\G)$ is unital if and only if $G^{(0)}$ is compact, in which case the indicator function of the unit space is the unit. 

As in the $C^*$-algebraic case, we have the following result.
\begin{lemma} [{\cite[Lemma 4.5]{CGTRigidityResultsForLpOp}}] \label{lem: basic estimaste involving sup-norm, operator-norm and I-norm}
	For any $f \in C_c (\G)$, we have 
	\begin{equation*}
		\lVert f \rVert_{\infty} \leq \lVert f \rVert_{\Fp(\G)} \leq \lVert f \rVert_I .
	\end{equation*}
	Moreover, for $f \in C_c (\G^{(0)})$, we have $\lVert f \rVert_{\infty} = \lVert f \rVert_{\Fp(\G)} = \lVert f \rVert_I$. 
\end{lemma}

Let $u \in \G^{(0)}$ and $q \in (1, \infty]$ be the Hölder conjugate to $p$, so that $\frac{1}{p} + \frac{1}{q} = 1$. We identify the dual of $\ell^p (\G_u)$ with $\ell^{q} (\G_u)$, where the dual pairing is given by 
\begin{equation*}
	\langle \gamma , \eta \rangle = \sum_{x \in \G_u} \gamma(x) \eta(x) ,
\end{equation*} 
for $\gamma \in \ell^p (\G_u)$ and $\eta \in \ell^q (\G_u)$. In the next proposition, by $\delta_x$ we mean the standard basis element in $\ell^p (\G_{s(x)})$ corresponding to $x \in \G_{s(x)}$. The space of continuous functions on $\G$ vanishing at infinity will be written as $C_0 (\G)$.

\begin{proposition}[{\cite[Proposition 4.7 and 4.9]{CGTRigidityResultsForLpOp}}] \label{prop: identification of Fp functions with C0 functions}
	The map $j_p \colon F_{\lambda}^{p}(\G) \to C_0 (\G)$ given by 
	\begin{equation*}
		j_p (a) (x) = \langle \lambda_{s(x)} (a) (\delta_{s(x)}) , \delta_{x} \rangle ,
	\end{equation*}
	for $a \in F_{\lambda}^p (\G)$ and $x \in \G$, is contractive, linear, injective and extends the identity on $C_c (\G)$. Moreover, given $a,b \in \Fp(\G)$ and $x \in \G$, we have that $j_p (a b) (x) = j_p (a) \ast j_p (b) (x)$, where the sum defining $j_p (a) \ast j_p (b) (x)$ is absolutely convergent.
\end{proposition}

Thus, under the identification provided by the map $j_p$, the algebraic operations on $\Fp(\G)$ may be expressed in the same way as the algebraic operations on $C_c (\G)$. In the $C^*$-algebraic case one also defines the involution on $C_c (\G)$ as $f^{\ast}(x) := \overline{f(x^{-1})}$, and this extends to an involution on $\Cred(\G)$. In our more general case, it extends to an isometric anti-isomorphism between $\Fp(\G)$ and $\Fq(\G)$, as the next lemma shows.

\begin{lemma} \label{lem: involution extends to isometric anti-homomorphism}
	Assume $p \in (1, \infty)$. For $f \in C_c (\G)$, we have that $\lVert f \rVert_{\Fp(\G)} = \lVert f^{\ast} \rVert_{\Fq(\G)}$, and the assignment $f \mapsto f^{\ast}$ extends to an isometric anti-isomorphism $^{\ast} \colon \Fp(\G) \to \Fq(\G)$. Moreover, for any $a \in \Fp(\G)$, $j_q (a^{\ast}) = j_p (a)^{\ast} $ as elements in $C_0 (\G)$.
	\begin{proof}
		
		The first statement can be proved as in \cite[Lemma 3.5]{AreOgEduardHermitianBanach...}. If $f_n \to a$ in $\Fp(\G)$, also $f_{n}^{\ast} \to a^{\ast}$ in $\Fq(\G)$. Since $j_p (f) = f = j_q (f) $, for all $f \in C_c (\G)$, and $^{\ast}$ is continuous on $C_0 (\G)$, we obtain 
		\begin{equation*}
			j_q (a^{\ast}) = \lim_{n \to \infty} j_q (f_{n}^{\ast}) = \lim_{n \to \infty} f_{n}^{\ast} = \lim_{n \to \infty} j_p (f_n)^{\ast} = (\lim_{n \to \infty} j_p (f_n))^{\ast} = j_p (a)^{\ast} .
		\end{equation*}
	\end{proof}
\end{lemma}

The next lemma roughly states that, under the identification provided by the map $j_p$, we may view elements in $\Fp(\G)$ as convolution operators induced from the corresponding $C_0 (\G)$ functions.

\begin{lemma} \label{lem: interpretation of Fp elements as convolution operators}
	Let $a \in \Fp(\G)$ be arbitrary and fix any unit $u \in \G^{(0)}$. We have that $\lambda_u (a) (\xi) = j_p (a) \ast \xi$, for any $\xi \in C_c(\G_u)$.
	\begin{proof}
		Let $u \in \G^{(0)}$ be any unit and suppose $f_n \to a$ in $\Fp(\G)$. Then by Proposition \ref{prop: identification of Fp functions with C0 functions}, $f_n \to j_p (a)$ in $C_0 (\G)$ and $\lambda_u (f_n) \to \lambda_u (a) $, as $n \to \infty$. Let $\xi \in C_c (\G_u)$ be given. Then since $\xi$ has finite support, we have 
		\begin{equation*}
			\lambda_u (a)(\xi) (x) = \lim_{n \to \infty} \lambda_u(f_n)(\xi) (x) = \lim_{n \to \infty} f_n \ast \xi (x) = j_p (a) \ast \xi (x),
		\end{equation*}
		for any $x \in \G_u$. 
	\end{proof}
\end{lemma}
For an overview of the theory of $L^p$-operator algebras, we refer the reader to \cite{GardellaModernLookofLp}.

\section{Rapid Decay and Polynomial Growth} \label{sec: Length Functions Rapid Decay and Polynomial Growth}
This section deals with length functions on groupoids, and from these the notions of rapid decay and polynomial growth of groupoids, which have been extensively studied in the group case. 
In the case of a group $\Gamma$, a length function on $\Gamma$ is a map $l \colon \Gamma \to \R_+$ such that $l(e) = 0$, where $e$ is the identity element, $l(xy) \leq l(x) + l(y)$ and $l(x^{-1}) = l(x)$. The natural generalization of this to groupoids is to 
change the first condition to $l(u) = 0$, for all $u \in \G^{(0)}$, and the second condition to $l(xy) \leq l(x) + l(y)$, for all $(x,y) \in \G^{(2)}$ (see Definition \ref{def:length-function-on-groupoid}).  

Length functions on groupoids have already been studied by several different authors. For instance, they appear in the study of amenability of measured groupoids in the book of Renault and Anantharaman-Delaroche \cite{AmenabilityGroupoidsRenault}; there, the length function is used to show that if certain growth conditions with respect to the length function are satisfied, the measured groupoid is amenable. Ma and Wu in \cite{ma2020almostElem} and \cite{ma2021fiberwise} show that length functions with the additional requirement that they are zero only on units, are in one-to-one correspondence with extended metrics on the groupoids. These metrics are then related to properties such as almost elementariness, fiberwise amenability and soficity of the topological full groups associated with groupoids. Some of these length functions studied by Ma and Wu were also used by Jiang, Zhang and Zhang in \cite{JiangZhangZhang:QuasiLocalityForEtaleGroupoids} to prove that for amenable $\sigma$-compact étale groupoids $\G$, the reduced groupoid $C^*$-algebra agrees with certain $C^*$-algebras consisting of $\G$-equivariant adjointable operators on a Hilbert $C^*$-module naturally associated with $\G$. 
In \cite{hou2017spectral}, Hou uses length functions as a tool to prove the existence of dense spectral invariant Fréchet algebras, which is one of the main motivations for their study in the group case. 

Like Hou, we shall primarily be interested in using length functions on groupoids to create dense spectral invariant Fréchet algebras in the reduced $L^p$-operator algebras $\Fp(\G)$. In particular, we will in the next section investigate the following question.
\begin{question} \label{question: motivating question}
	For which types of étale groupoids $\G$ do we have isomorphisms in $K$-theory $K_{\ast}(\Fp(\G)) \cong K_{\ast}(\Fq(\G))$, for some (or all) $p,q \in [1, \infty)$?
\end{question}
Throughout the rest of the text, we shall always assume that $\G$ is an étale groupoid and that $p \in [1, \infty)$.

\begin{definition}\label{def:length-function-on-groupoid}
	A \emph{length function} on an étale groupoid $\G$ is a map $l \colon \G \to \R_+$, satisfying
	\begin{itemize}
		\item[(1)] $l(u) = 0$, for all $u \in \G^{(0)}$;
		\item[(2)] $l(xy) \leq l(x) + l(y)$, for all $(x,y) \in \G^{(2)}$;
		\item[(3)] $l(x^{-1}) = l(x)$, for all $x \in \G$.  
	\end{itemize}
\end{definition}
A length function is said to be \emph{locally bounded} if it is bounded on compact subsets.

The analogue to the $\ell^1 $-norm of an element in the group ring is the $I $-norm of a continuous compactly supported function on an étale groupoid, defined in Section \ref{subsec: preliminaries on reduced Lp operator algebras}. The natural analogue then for the $\ell^p $-norm is the $I^p $-norm.

\begin{definition} \label{def: I^p and p,k -norm defined}
	Let $\G$ be an étale groupoid and let $p \in [1, \infty)$. For $f \in C_c (\G)$, we define the $I^p $-norm to be 
	\begin{equation*}
		\lVert f \rVert_{I^p} := \max \Big\lbrace \sup_{u \in \G^{(0)}} \Big( \sum_{x \in \G_u} \Big| f(x)\Big|^p \Big)^{1/p} , \sup_{u \in \G^{(0)}} \Big( \sum_{x \in \G^{u}} | f(x)|^p \Big)^{1/p} \Big\rbrace.
	\end{equation*}
	Assuming $\G$ is endowed with a locally bounded length function $l$, we define 
	\begin{equation*}
		\lVert f \rVert_{p,k} := \big\lVert f(1+l)^k \big\rVert _{I^p} ,
	\end{equation*}
	for $f \in C_c (\G)$ and $k \in \N_0 $. 
\end{definition}

It is clear that $\lVert f \rVert_{\infty} \leq \lVert f \rVert_{I^p} = \lVert f \rVert_{p,0} \leq \lVert f \rVert_{p,k}$, for all $f \in C_c (\G)$ and all $k \in \N_0$. Thus, $\lVert \cdot \rVert_{I^p} $ is a norm on $C_c(\G)$, and when $\G$ has a locally bounded length function, so is $\lVert \cdot \rVert_{p,k} $, for each $k \in \N_0 $. 

\begin{lemma} \label{lem: The Banach space Lpk are contained in C0}
	Let $p \in [1, \infty)$ and $k \in \N_0 $. Assume $\G$ is an étale groupoid endowed with a locally bounded length function $l$. Let $L_{p,k}(\G)$ be the subspace consisting of $f \in C_0 (\G)$ such that $\lVert f \rVert_{p,k} < \infty$. Then $L_{p,k} (\G)$ is a Banach space.
	\begin{proof}
		First of all, we clearly also have that $\lVert f \rVert_\infty \leq \lVert f \rVert_{p,k}$ for all $f \in L_{p,k}(\G)$. Let $\{f_n \}_n \subset L_{p,k} (\G)$ be a Cauchy-sequence under the norm $\lVert \cdot \rVert_{p,k}$. In particular, $\{f_n \}_n$ is a Cauchy sequence in $C_0 (\G)$, and so there exists $f \in C_0 (\G)$ such that $\lim_{n \to \infty} f_n = f$ in $C_0 (\G)$. For any $u \in \G^{(0)} $ and finite subsets $\mathcal{F}_{u} \subset \G_u$ and $\mathcal{F}^{u} \subset \G^u$, we have that 
		\begin{equation*}
			\sum_{x \in \mathcal{F}_{u}} | f(x) |^p (1+l(x))^{pk} = \lim_{n \to \infty} \sum_{x \in \mathcal{F}_{u}} | f_n (x) |^p (1+l(x))^{pk} \leq \sup_{n} \lVert f_n \rVert_{p,k} < \infty ,
		\end{equation*}
		and similarly, 
		\begin{equation*}
			\sum_{x \in \mathcal{F}^{u}} | f(x) |^p (1+l(x))^{pk} \leq \sup_{n} \lVert f_n \rVert_{p,k} < \infty ,
		\end{equation*}
		for any $k \in \N_0 $. Therefore, $\lVert f \rVert_{p,k} \leq \sup_{n} \lVert f_n \rVert_{p,k} < \infty $, and so $f \in L_{p,k}(\G)$. A priori, it is not immediately clear that $f_n \to f$ in $L_{p,k} (\G)$, but we will show in the sequel that this is indeed the case. Notice that for any $u \in \G^{(0)}$, we have that 
		\begin{equation*}
			\max \{ \lVert f_n (1 + l)^k  - f_m (1 + l)^k  \rVert_{\ell^p (\G_u)} , \lVert f_n (1 + l)^k - f_m (1 + l)^k  \rVert_{\ell^p (\G^{u})} \} \leq \lVert f_n - f_m \rVert_{p,k} \to 0 ,
		\end{equation*}
		as $m,n \to \infty$, independently of $u \in \G^{(0)}$. Thus, for each $u \in \G^{(0)}$ there exist $g_u \in \ell^p (\G_u)$ and $g^{u} \in \ell^p (\G^{u})$ such that $(1+l)^k f_n \to g_u$ in $\ell^p (\G_u)$ and $(1 + l)^k f_n \to g^u$ in $\ell^p (\G^u)$. It follows that there exists a subsequence $n_j$ such that $(1 + l)^k f_{n_j} \to g_u$ and $(1 + l)^k f_{n_j} \to g^u$ pointwise in respectively $\G_u$ and $\G^u$. But we know that $(1 + l)^k f_{n_j} \to (1 + l)^k f$ pointwise, so $g_u = (1 + l)^k f \big|_{\G_u}$ and $g^u = (1 + l)^k f \big|_{\G^{u}}$. In particular, $ (1+l)^k f_n \to (1 + l)^k f$ in both $\ell^p (\G_u)$ and $\ell^p (\G^u)$. Now, given $\epsilon > 0$, there exists $N \in \N$ such that whenever $n,m \geq N$, we have 
		\begin{equation*}
			\max \Big\{ \big\lVert f_n (1 + l)^k  - f_m (1 + l)^k  \big\rVert_{\ell^p (\G_u)} , \big\lVert f_n (1 + l)^k - f_m (1 + l)^k  \big\rVert_{\ell^p (\G^{u})} \Big\} \leq \epsilon ,
		\end{equation*}
		for all $u \in \G^{(0)}$. Thus,
		\begin{align*}
			&\max \Big\{ \big\lVert f (1 + l)^k  - f_m (1 + l)^k  \big\rVert_{\ell^p (\G_u)} , \big\lVert f (1 + l)^k - f_m (1 + l)^k  \big\rVert_{\ell^p (\G^{u})} \Big\}  \\
			&= \lim_{n \to \infty} \max \Big\{ \big\lVert f_n (1 + l)^k  - f_m (1 + l)^k  \big\rVert_{\ell^p (\G_u)} , \big\lVert f_n (1 + l)^k - f_m (1 + l)^k  \big\rVert_{\ell^p (\G^{u})} \Big\} \leq \epsilon ,
		\end{align*} 
		when $m \geq N $, for any $u \in \G^{(0)}$. That is, $\lVert f - f_m \rVert_{p,k} \leq \epsilon $, when $m \geq N$. This means that $f_n \to f$ in $L_{p,k} (\G)$.
	\end{proof}
\end{lemma}

If $\G$ is an étale groupoid endowed with a locally bounded length function, then we clearly have that 
\begin{equation*}
	C_c (\G) \subset L_{p,k+1} (\G) \subset L_{p,k}(\G) \subset C_0 (\G) ,
\end{equation*}
for all $k \in \N_0 $. Let
\begin{equation*}
	S_{p}^{l}(\G) := \bigcap_{k = 0}^{\infty} L_{p,k}(\G) .
\end{equation*}
Then $S_{p}^{l} (\G) \subset C_0 (\G)$ and is a Fréchet space under the locally convex topology determined by the increasing family of norms $\{ \lVert \cdot \rVert_{p,k} \}_{k \in \N_0}$. Adopting the usual terminology from  the group case, we call $S_{p}^{l}(\G)$ the space of rapidly decreasing functions on $\G$ with respect to the locally bounded length function $l$. 

\begin{definition} \label{def: RD definition}
	Let $p \in [1, \infty)$. An étale groupoid $\G$ is said to have \emph{property $RD_p$ with respect to a locally bounded length function} $l$ if there exist a positive constant $C > 0$ and $k \in \N_0$ such that for all $f \in C_c (\G)$, we have $\lVert f \rVert_{\Fp(\G)} \leq C \lVert f \rVert_{p,k} = C \big\lVert f(1+l)^k \big\rVert _{I^p}$.
	
	We shall say that an étale groupoid has \emph{property $RD_p$} if it has property $RD_p$ with respect to some locally bounded length function.
\end{definition}
As explained in \cite[Section 3]{hou2017spectral}, if a length function $l_1$ polynomially dominates another length function $l_2$, in the sense that there is $c > 0$ and $k \in \N$ such that $l_2 (x) \leq c (1 + l_1 (x))^k$ for all $x \in \G$, then $\G$ has property $RD_p$ with respect to $l_1$ if it has property $RD_p$ with respect to $l_2$. Moreover, if $\G$ is a compactly generated groupoid, meaning that $\G = \bigcup_{n = 1}^{\infty} K^n$ for some symmetric compact set $K \subset \G$, then the canonically associated word length function given by $l_K (x) := \min\{ n \colon x \in \bigcup_{i = 1}^{n} K^i  \}$ when $x \notin \G^{(0)}$ and $l_K (u) := 0$ for $u \in \G^{(0)}$, is easily seen to dominate any locally bounded length function on $\G$. In particular, if $\G$ is compactly generated and has property $RD_p$, then $\G$ has property $RD_p$ with respect to any locally bounded word length function. 

When $p = 2$, Definition \ref{def: RD definition} is the same as \cite[Definition 3.2]{hou2017spectral}, and it naturally generalizes the well known property $RD$ for discrete groups. We refer the reader to \cite{SurveyOnRDChatterji, ConnectedLieGroupsAndRD, ChatterjoSomegreometricGroups, kTheoryAndRDprop, JolissaintRD, kTheoryOfBAandRD} for information and examples of such groups.

Like in the group case, there are several equivalent definitions one can give for property $RD_p$. First of all, the argument in \cite[Lemma 3.3]{hou2017spectral} generalizes immediately to the statement that when $\G$ has property $RD_p$, there exist for each $f \in S_{p}^{l}(\G)$ an element $a \in \Fp(\G)$ such that $j_p (a) = f $, and thus $S_{p}^{l}(\G)$ is included in $\Fp(\G)$ via $\iota = j_{p}^{-1} $. Conversely, when $\iota = j_{p}^{-1}$ is an inclusion, it is a closed map and this ensures that $\G$ has property $RD_p$. We record this in the next proposition, giving the first equivalent definition of $RD_p $.

\begin{proposition} \label{prop: Frechet space included in Fp if and only if RD}
	Let $\G$ be an étale groupoid and let $p \in [1, \infty) $. Then $\G$ has property $RD_p$ if and only if $S_{p}^{l}(\G)$ is contained in $\Fp(\G)$.
\end{proposition}
There is yet another characterization, at least when the length function is continuous. Given a length function $l$ on a groupoid $\G$, we define for $R > 0$, the balls $B_l (R) := \{ x \in \G \colon l(x) \leq R \}$. 
\begin{proposition} \label{prop: RD equivalent to functions supported in R-balls...}
	Let $\G$ be an étale groupoid endowed with a locally bounded length function $l$ and let $p \in [1, \infty)$. If $\G$ has property $RD_p$, then there exist positive constants $C,D > 0$ such that for all $f \in C_c (G)$ with support contained in $B_l (R) $, where $R \geq 1$, we have 
	\begin{equation*}
		\lVert f \rVert_{\Fp(\G)} \leq C R^D \lVert f \rVert_{I^p} .
	\end{equation*}
	The converse holds if $l$ is continuous.	
	\begin{proof}
		Assume first that $\G$ has property $RD_p $, for constants $C > 0$ and $k \in \N_0$ as in Definition \ref{def: RD definition}. Let $f \in C_c(\G)$ be supported in $B_l (R)$, for some $R \geq 1$. We have that $\lVert f \rVert_{\Fp(\G)} \leq C \lVert f \rVert_{p,k} $. Fix any unit $u \in \G^{(0)}$; then 
		\begin{align*}
			C \Big(\sum_{x \in \G_u} |f(x)|^{p} (l(x) + 1)^{pk}\Big)^{1/p} & \leq C \Big(\sum_{x \in \G_u} |f(x)|^p (R + 1)^{pk}\Big)^{1/p}  \\ 
			& \leq 2^k R^k C \Big( \sum_{x \in \G_u} |f(x)|^p \Big)^{1/p}.
		\end{align*}
		The same estimate holds when summing over $\G^u$. Thus, after taking the supremum over all $u \in \G^{(0)}$, we obtain 
		\begin{equation*}
			\lVert f \rVert_{\Fp(\G)} \leq C \lVert f \rVert_{p,k} \leq 2^k R^k C \lVert f \rVert_{I^p}.
		\end{equation*} 
		
		\
		
		If the length function is continuous, the sets $U_n = l^{-1}(n, n+2)$, for $n \geq 1$, and $U_0 = l^{-1} ([0,2))$, are open sets such that $\G = \bigcup_{n = 0}^{\infty} U_n$. Assume that $C, D > 0$ are the constants for which 
		\begin{equation*}
			\lVert f \rVert_{\Fp(\G)} \leq C R^D \lVert f \rVert_{I^p} ,
		\end{equation*}
		for all $f \in C_c (\G)$ supported in $B_l (R)$, with $R \geq 1$. Let $f \in C_c(\G)$ be arbitrary. There exists $N \in \N$ such that $\supp (f) \subset \bigcup_{n = 0}^{N} U_n $, and $f$ vanishes on $l^{-1}(N+1 , \infty)$. Let $\{\rho_n \}_{n = 0}^{N} \cup \{ \xi \}$ be a partition of unity subordinate to the finite open cover $\{ U_n \}_{n = 0}^{N} \cup \{ l^{-1} ( (N+1 , \infty) ) \} $ of $\G $. Thus, $\xi + \sum_{n = 0}^{N} \rho_n = 1 $, $\supp (\rho_n) \subset U_n $, $\supp (\xi) \subset l^{-1} ( (N+1 , \infty) ) $, and the functions $\xi$ and $\rho_n$ are continuous, for $0 \leq n \leq N$. We may write 
		\begin{equation*}
			f = f \cdot \Big(\xi + \sum_{n = 0}^{N} \rho_n \Big)  = \sum_{n = 0}^{N} f \rho_n ,
		\end{equation*}
		and $f \rho_n $ is supported in $B_l (n+2)$, for $0 \leq n \leq N$; thus 
		\begin{align*}
			\lVert f \rVert_{\Fp (\G)} &\leq \sum_{n = 0}^{N} \lVert f \rho_n \rVert_{\Fp(\G)} \leq C \sum_{n = 0}^{N} (n + 2)^D \lVert f \rho_n \rVert_{I^p} \leq 2^D C \sum_{n = 0}^{N} (n+1)^D \lVert f \rho_n \rVert_{I^p} \\ 
			&= 2^D C \sum_{n = 0}^{N} (n +1)^{-2} \big\lVert f \rho_n (n+1)^{D+2} \big\rVert_{I^p} \leq 2^D C \sum_{n = 0}^{N} (n +1)^{-2} \big\lVert f \rho_n (1+l)^{D+2} \big\rVert_{I^p} \\
			&\leq 2^D C \sum_{n = 0}^{N} (n +1)^{-2} \big\lVert f (1 + l)^{D+2} \big\rVert_{I^p} \leq 2^D C \Big( \sum_{n = 0}^{\infty} (n+1)^{-2} \Big) \big\lVert f \big\rVert_{p, D+2}
		\end{align*}
		In the estimate above, we used in the fifth step that $n \leq l$ on the support of $f \rho_n$ and in the sixth that $\rho_n \leq 1$, for $0 \leq n \leq N$. 
	\end{proof}
\end{proposition}
In the case of discrete groups, the equivalent statement in Proposition \ref{prop: RD equivalent to functions supported in R-balls...} is the definition of $RD_p$ used in \cite{kTheoryOfBAandRD}. An interesting observation is that property $RD_p$ is automatic if $RD_d$ holds for some $d \geq p$. This is shown in the next proposition, whose argument is a generalization of \cite[Theorem 4.4]{kTheoryOfBAandRD}.
\begin{proposition}\label{prop : RD inherited below (étale case)}
	If an étale groupoid $\G$ has property $RD_d $, for some $d \in [1, \infty)$, then $\G$ has property $RD_{p}$, for all $1 \leq p \leq d$. 
	\begin{proof}
		Put $\alpha = p/d$. For any $f \in C_c (\G)$, we define $f_{\alpha} (x) = | f(x) |^{\alpha}$. Then for any $u \in \G^{(0)}$ and any $f \in C_c (\G)$, we have $\lVert f_{\alpha} \rVert_{\ell^{d} (\G_u)}^{d} = \lVert f \rVert_{\ell^{p}(\G_u)}^{p}$, and $\lVert f_{\alpha} \rVert_{\ell^{d} (\G^{u})}^{d} = \lVert f \rVert_{\ell^{p}(\G^{u})}^{p}$. Since $\alpha \leq 1$, we have for any finite collection $a_i \geq 0$ that $(a_1 + \dots + a_n)^{\alpha} \leq a_{1}^{\alpha} + \dots + a_{n}^{\alpha}$, and so for $\xi \in C_c (\G_u)$, we have the inequality 
		\begin{align*}
			(| f | \ast | \xi |)_{\alpha} (x) &= ( | f | \ast |\xi | (x) )^{\alpha} = \Big( \sum_{y \in \G_{u}} | f | (x y^{-1}) | \xi | (y) \Big)^{\alpha} \\ 
			&\leq \sum_{y \in \G_{u}} | f |^{\alpha} (x y^{-1}) | \xi |^{\alpha} (y) = f_{\alpha} \ast \xi_{\alpha} (x).
		\end{align*}
		Using the above and the fact that $\G$ has $RD_d$, we obtain for any $\xi \in C_c (\G_u)$ that
		\begin{align*}
			\lVert f \ast \xi \rVert_{\ell^{p}(\G_u)}^{p} &\leq \lVert | f | \ast |\xi | \rVert_{\ell^{p}(\G_u)}^{p} = \lVert (| f | \ast | \xi |)_{\alpha} \rVert_{\ell^{d}(\G_u)}^{d} \\ 
			&\leq \lVert f_{\alpha} \ast \xi_{\alpha} \rVert_{\ell^{d}(\G_u)}^{d} \leq \lVert f_{\alpha} \rVert_{B(\ell^d (\G_u))}^{d} \lVert \xi_{\alpha} \rVert_{\ell^{d}(\G_u)}^{d} \\ 
			&\leq C^d \lVert f_{\alpha} (1 + l)^k \rVert_{I_d}^{d} \lVert \xi_{\alpha} \rVert_{\ell^{d}(\G_u)}^{d} = C^d \lVert f_{\alpha} (1 + l)^k \rVert_{I_d}^{d} \lVert \xi \rVert_{\ell^{p}(\G_u)}^{p} ,
		\end{align*}
		where $C > 0$ and $k \in \N_0$ are the constants from $\G$ having $RD_d$. From this we get that $ \lVert f \rVert_{F_{\lambda}^{p}(\G)}^{p} \leq C^d \lVert f_{\alpha} (1 + l)^k \rVert_{I_d}^{d} $. If we let $n \in \N$ be such that $n p \geq d$, then it is easy to see that $ \lVert f_{\alpha} (1+l)^k \rVert_{I_d}^{d} \leq \lVert f (1 + l)^{nk} \rVert_{I^p}^{p} $, and thus we get
		\begin{equation*}
			\lVert f \rVert_{F_{\lambda}^{p}(\G)} \leq C^{d/p} \lVert f (1 + l)^{nk} \rVert_{I^p} ,
		\end{equation*}
		showing that $\G$ has property $RD_p$.
	\end{proof}
\end{proposition}

With respect to natural length functions, property $RD_p$ for transformation groupoids formed from the data of compact Hausdorff spaces and discrete groups, implies property $RD_p$ for the associated discrete groups. This was also proved independently in \cite[Proposition 5.3]{weygandt2023rapid} when $p = 2$. Let us first recall the definition of a transformation groupoid.
\begin{definition}
	Let $X$ be a compact Hausdorff space and $\Gamma$ a discrete group. Suppose there is an action of $G$ on $X$. We write $\gamma \cdot x$ for the action at $(\gamma , x) $, where $\gamma \in \Gamma$ and $x \in X $. The transformation groupoid is denoted by $X \rtimes \Gamma$. As a topological space, $X \rtimes \Gamma$ is just $X \times \Gamma$, and the groupoid operations are given by $(\gamma \cdot x, \tau) (x, \gamma) = (x, \tau \gamma)$, and $(x, \gamma)^{-1} = (\gamma \cdot x, \gamma^{-1})$. Thus, the unit space is $X \times \{e\} $, where $e$ is the identity element of $\Gamma$, and the source and range maps become respectively $s(x, \gamma) = (x,e)$, $r(x, \gamma) = (\gamma \cdot x , e)$. In this way, $X \rtimes \Gamma$ becomes an étale groupoid, and its unit space is identified with $X$ in the obvious way. Notice that the source fiber at $x \in X$ is given by $(X \rtimes \Gamma)_x = \{ (x, \gamma) \colon \gamma \in \Gamma \}$.
\end{definition}

\begin{proposition}
	Let $p \in [1, \infty)$, $X$ a compact Hausdorff space and let $\Gamma$ be a discrete group acting on $X$. Assume that $\Gamma$ is endowed with a locally bounded length function $l \colon \Gamma \to \R_+$. The transformation groupoid $\G = X \rtimes \Gamma$ then has a natural length function $(x, \gamma) \mapsto l(\gamma)$ which we still denote by $l$. With respect to $l$, if the transformation groupoid $\G$ has $RD_p $, then $\Gamma$ has property $RD_p $.
	\begin{proof}
		Suppose that the transformation groupoid $\G$ has property $RD_p$ with constants $C > 0$ and $k \in \N_0$ as in Definition \ref{def: RD definition}. Fix any positive $f \in \C \Gamma$, and define $\tilde{f} (x, \gamma) = f(\gamma)$. Since $X$ is compact, $\tilde{f}$ is continuous compactly supported on $\G$, and so 
		\begin{equation*} \label{eq: groupoid norm bounded by group Ip norm}
			\lVert \tilde{f} \rVert_{\Fp(\G)} \leq C \big\lVert \tilde{f} (1 + l)^k \big\rVert_{I^p} = C \big\lVert f (1 + l)^k \big\rVert_{\ell^p (\Gamma)} .
		\end{equation*}
		Fixing $x \in X $, if $g_x \in \ell^p(\G_x) $ is identified with $g \in \ell^p (\Gamma)$, we have that  
		\begin{align*}
			\lambda_x (\tilde{f}) (g_x) (x,\gamma) &= \sum_{\mu \in \Gamma} \tilde{f}((x,\gamma) (x, \mu)^{-1}) g_x (x, \mu) = \sum_{\mu \in \Gamma} \tilde{f}((x,\gamma) (\mu \cdot x, \mu^{-1})) g_x (x, \mu) \\
			&= \sum_{\mu \in \Gamma} \tilde{f}((\mu \cdot x, \gamma \mu^{-1})) g_x (x, \mu) = \sum_{\mu \in \Gamma} f(\gamma \mu^{-1}) g(\mu) = \lambda(f) (g) (\gamma) ,
		\end{align*}
		where $\lambda$ is the usual left regular representation of the group $\Gamma $. Since for any such $g_x \in \G_x $, we have that $\lVert g_x \rVert_{\ell^p (\G_x)} = \lVert g \rVert_{\ell^p (\Gamma)} $, we see that 
		\begin{equation*}
			\lVert \lambda_x (\tilde{f}) \rVert = \sup_{\lVert g_x \rVert_{\ell^p (\G_x)} \leq 1} \lVert \lambda_x (\tilde{f}) (g_x) \rVert_{\ell^p (\G_x)} = \sup_{\lVert g \rVert_{\ell^p (\Gamma)} \leq 1} \lVert \lambda(f) (g) \rVert_{\ell^p (\Gamma)} = \lVert \lambda (f) \rVert ,
		\end{equation*}
		and so
		\begin{equation*}
			\lVert f \rVert_{\Fp(\Gamma)} = \lVert \lambda(f) \rVert_{B(\ell^p (\Gamma))} = \sup_{x \in X} \lVert \lambda_x (\tilde{f}) \rVert_{B(\ell^p (\G_x))} = \lVert \tilde{f} \rVert_{\Fp(\G)} \leq C \big\lVert f (1 + l)^k \big\rVert_{\ell^p (\Gamma)} ,
		\end{equation*}
		showing that $\Gamma$ has property $RD_p $. 
	\end{proof} 
\end{proposition}

Next, we shall define what it means for an étale groupoid to have polynomial growth with respect to a length function. The definition we employ naturally generalizes the same notion for groups, and is taken from \cite{hou2017spectral} (see also \cite{NekrashevychGrowthOfEtalegGroupoids}). 

Fix a unit $u \in \G^{(0)}$, let $m \geq 0$, and define the sets 
\begin{equation*}
	B_{\G^u} (m) = \{ x \in \G^{u} \colon l(x) \leq m \} \text{ and } B_{\G_u} (m) = \{ x \in \G_u \colon l(x) \leq m \} .
\end{equation*}
Denoting by $| A |$ the cardinality of a set $A$, notice that since $l(x^{-1}) = l(x) $, for all $x \in \G$, we have $| B_{\G_u} (m) | = | B_{\G^u} (m) |$.

\begin{definition} \cite[Definition 3.1]{hou2017spectral} \label{def: polynomial growth}
	We say that an étale groupoid $\G$ is of \emph{polynomial growth with respect to a length function} $l$ if there are constants $c \geq 1$ and $r \geq 1$ such that for each $m \geq 0$, we have 
	\begin{equation*}
		\sup_{u \in \G^{(0)}} | B_{\G_u} (m)| = \sup_{u \in \G^{(0)}} | B_{\G^{u}} (m) | \leq c (1 + m)^r .
	\end{equation*}
	An étale groupoid $\G$ is said to be of \emph{polynomial growth} if it is of polynomial growth with respect to some length function.
\end{definition}

Étale groupoids of polynomial growth enjoy the following permanence properties outlined in the next lemma. We omit the proofs, since they are straightforward. 

\begin{lemma} \label{lem: permanence properties of polynomial growth}
	\
	\begin{itemize}
		\item[(1)] The étale groupoids $\G_1$ and $\G_2$ have polynomial growth with respect to length functions $l_1$ and $l_2$ respectively if and only if the groupoid $\G_1 \sqcup \G_2$ has polynomial growth with respect to the length function $l(g) = l_1 (g)$ if $g \in \G_1$ and $l(g) = l_2 (g)$ if $g \in \G_2$. 
		\item[(2)] If $\mathcal{H} \subset \G$ is a subgroupoid of an étale groupoid $\G$ which has polynomial growth with respect to $l$, then $\mathcal{H}$, with the restriction of $l$ as length function, has polynomial growth.
		\item[(3)] The étale groupoids $\G_1$ and $\G_2$ have polynomial growth with respect to length functions $l_1$ and $l_2$ respectively if and only if the étale groupoid $\G_1 \times \G_2$ has polynomial growth with respect to the length function $l(g,h) = l_1 (g) + l_2 (h)$.
		\item[(4)] If $\phi \colon \G \to \mathcal{H}$ is a bijective groupoid homomorphism between étale groupoids, and $\mathcal{H}$ has polynomial growth with respect to a length function $l$, then $\G$ has polynomial growth with respect to the length function $l \circ \phi \colon \G \to \R_+$.
	\end{itemize}
\end{lemma}

That polynomial growth is stable under Kakutani equivalence is a little less straightforward, so we include an argument for this.
\begin{lemma} \label{lem: polynomial growth stable under kakutani equivalence}
	For ample groupoids with $\sigma $-compact unit spaces, polynomial growth is stable under Kakutani equivalence. 
	\begin{proof}
		Assume that $\G$ is of polynomial growth, that $\G$ is Kakutani-equivalent to $\mathcal{H}$, and that both are ample groupoids with $\sigma$-compact unit spaces. By \cite[Theorem 3.12]{AmpleGroupoidsAndHomology}, 
		\begin{equation*}
			\G \times \mathcal{R} \cong \mathcal{H} \times \mathcal{R} ,
		\end{equation*}
		as étale groupoids. Here $\mathcal{R}$ is the full equivalence relation on $\N$. With the length function $(n,m) \mapsto |n - m |$, $\mathcal{R}$ has polynomial growth. By Lemma \ref{lem: permanence properties of polynomial growth}, $\G \times \mathcal{R}$ has polynomial growth, and then by the same lemma, so does $\mathcal{H} \times \mathcal{R}$. Since $\mathcal{H}$ can be identified as a subgroupoid of $\mathcal{H} \times \mathcal{R} $, Lemma \ref{lem: permanence properties of polynomial growth} applies once more to give that $\mathcal{H}$ has polynomial growth.
	\end{proof}
\end{lemma}

Polynomial growth can in fact be seen as a strong form of rapid decay, as is shown in the next proposition. This is \cite[Proposition 3.5]{hou2017spectral} adapted to our more general case.
\begin{proposition} \label{prop: polynomial growth implies RD_q for all q}
	If $\G$ is an étale groupoid which has polynomial growth with respect to a locally bounded length function $l$, then $\G$ has property $RD_p$ with respect to $l$, for every $p \in [1, \infty )$.
	\begin{proof}
		Let $f \in C_c (\G)$, $c \geq 1$ and let $r \geq 1$ be an integer such that $\sup_{u \in \G^{(0)}} | B_{\G_u} (m)| \leq c (1 + m)^r $. Any étale groupoid has $RD_1$. Let $p \in (1, \infty)$ and let $q$ be the conjugate exponent. Put $k := 2 + r$ and fix any unit $u \in \G^{(0)}$. We have that
		\begin{align*}
			\sum_{x \in \G_u} (1 + l(x))^{-qk} & \leq \sum_{n = 0}^{\infty} \sum_{\substack{x \in \G_u \\ n \leq l(x) \leq n+1}} (1 + l(x))^{-qk} \leq \sum_{n = 0}^{\infty} | B_{\G_u}(n+1) | (1 + n)^{-qk} \\
			& \leq c \sum_{n = 0}^{\infty} (2 + n)^r (1 + n)^{-qk} \leq 2^r c \sum_{n = 0}^{\infty} (1 + n)^{-2q} =: \tilde{c}. 
		\end{align*}
		By Hölder's inequality,
		\begin{align*}
			\sum_{x \in \G_u} | f(x) | &= \sum_{x \in \G_u} | f(x) | (1 + l(x))^k (1 + l(x))^{-k} \\
			&\leq \Big( \sum_{x \in \G_u} | f(x) |^p (1 + l(x))^{p k} \Big)^{1/p} \Big( \sum_{x \in \G_u} (1 + l(x))^{-q k} \Big)^{1/q} \leq \tilde{c}^{1/q} \lVert f \rVert_{p,k}.
		\end{align*}
		Moreover, we have
		\begin{align*}
			\sum_{x \in \G^u} |f(x) | = \sum_{x \in \G_u} | f^{\ast} (x) | \leq \tilde{c}^{1/q} \lVert f^{\ast} \rVert_{p,k} = \tilde{c}^{1/q} \lVert f \rVert_{p,k}.
		\end{align*}
		Since $u \in \G^{(0)}$ was arbitrary, using Lemma \ref{lem: basic estimaste involving sup-norm, operator-norm and I-norm}, we obtain
		\begin{equation*}
			\lVert f \rVert_{\Fp(\G)} \leq \lVert f \rVert_I \leq \tilde{c}^{1/q} \lVert f \rVert_{p,k} , 
		\end{equation*}
		and this shows that $\G$ has property $RD_p$. 
	\end{proof}
\end{proposition}

Some straightforward examples of étale groupoids which have polynomial growth are collected in the next example. 
\begin{example} \label{ex: easy examples of polynomial growth groupoids}
	\
	
	\begin{itemize}
		\item[(1)] Any compact étale groupoid has polynomial growth.
		\item[(2)] A locally compact Hausdorff space, seen as a trivial groupoid, has polynomial growth.
		\item[(3)] Any discrete group of polynomial growth is of polynomial growth when considered as a groupoid.
		\item[(4)] Suppose $\Gamma$ is a discrete group equipped with a length function $l \colon \Gamma \to \R_+$, and suppose $X$ is a locally compact Hausdorff space on which $\Gamma$ acts via homeomorphisms. The length function may be extended to the transformation groupoid via the map $(x,\gamma) \in X \rtimes \Gamma \mapsto l(\gamma)$. With respect to this length function, the transformation groupoid $X \rtimes \Gamma$ has polynomial growth if and only if the group $\Gamma$ has polynomial growth with respect to $l$.
		
	\end{itemize}
\end{example}

\begin{example}{(Topological full groups of polynomial growth)}
	
	Let $\G$ be an effective étale groupoid with compact unit space. A bisection $U\subseteq \G$ is called \emph{full} if $r(U)=s(U)=\G^{(0)}$. The \emph{topological full group of $\G$} is the group of all full bisections, and it is denoted by $[[\G]]$. 
	
	Suppose that $[[\G]]$ has polynomial growth. Then for every finite set $X=\{U_1,\ldots, U_m\}\subseteq [[\G]]$ that contains the identity and generates $[[\G]]$, there exist $R,C >0$ such that $|X^n|\leq C (1+n)^R$. Assume that $|r(\G_u)|\geq 2 $, for every $u \in \G^{(0)}$. Then by \cite[Lemma 3.9]{TopFullGroupsAmpleGroupoidsPetterOgEduard}, we have that $S=\bigcup_{i=1}^m U_i$ is a compact generating set of $\G$. Given $u \in \G^{(0)}$, we have that $|B_{\G_u}(n)|= |S^n u|\leq |X^n|\leq C (1+n)^R$. Thus, $\G$ has polynomial growth.     
\end{example}

A more involved example is that of AF-groupoids. Recall that an AF-groupoid $\G$ is an ample second-countable groupoid, such that $\G^{(0)}$ is a locally compact Cantor space, $\G = \bigcup_{n = 1}^{\infty} \mathcal{K}_n$, where each $\mathcal{K}_n$ is a principal clopen subgroupoid for which $\mathcal{K}_{n}^{(0)} = \G^{(0)}$, $\mathcal{K}_n \setminus \G^{(0)}$ is compact and $\mathcal{K}_n \subset \mathcal{K}_{n+1} $. When the AF-groupoid $\G$ has compact unit space, each $\mathcal{K}_n$ is compact. By 
\cite[Theorem 3.9]{SkauEtAlAffableEquivalenceRelations}, any AF-groupoid can, up to isomorphism, be constructed from a Bratteli diagram as described in \cite[Section 11.5]{TopFullGroupsAmpleGroupoidsPetterOgEduard}. For the convenience of the reader, we shall recall this construction in what comes next (see also \cite[Theorem 3.6]{SkauEtAlAffableEquivalenceRelations} and \cite[Example 2.2]{MatuiAbelSympAFGroupoidBD}): A Bratteli diagram $B = (V, E)$ consists of a disjoint union of finite sets of vertices $V = \bigsqcup_{n = 0}^{\infty} V_n $, a disjoint union of finite sets of edges $E = \bigsqcup_{n = 1}^{\infty} E_n$, and maps $i \colon E_n \to V_{n -1}$ and $t \colon E_n \to V_n $, for $n \geq 1$. Recall also that a source is a vertex $v \in V$ for which there is no edge with $v$ as target; that is, there is no $e \in E$ for which $t(e) = v$. Let $S(B)$ denote the set of all sources. The Bratteli diagram $B$ is called standard if $V_0 = \{v_0 \} = S(B)$. Suppose we are given a Bratteli diagram $B = (V, E)$, with set of sources $S(B)$. We can then construct its associated infinite path space. First of all, for any source on the nth level $v \in S(B) \cap V_n $, the set of infinite paths starting at $v$ is the set 
\begin{equation*}
	X_v := \{ e_{n+1} e_{n+2} \dots \colon e_i \in E_{i} , i(e_{n+1}) = v \text{ and } i(e_{n + k +1}) = t(e_{n+k}) \forall k \geq 1 \} .
\end{equation*}
For $x \in X_v$, we shall write $x = x_{n+1} x_{n+2} \dots$, where $x_{i} \in E_i$. The infinite path space associated to $B$ is then 
\begin{equation*}
	X_B := \bigsqcup_{v \in S(B)} X_v .
\end{equation*}
Its topology has as basis the cylinder sets given for a finite path $\mu$ with $i(\mu) \in S(B) \cap V_n$ as 
\begin{equation*}
	Z(\mu) := \{ e_{n+1} e_{n+2} \dots \in X_{i(\mu)} \colon e_{n+1} \dots e_{n + |\mu|} = \mu \} ,
\end{equation*}
where $| \mu|$ denotes the number of edges comprising the path $\mu$. These cylinder sets are compact open. Define, for each $N \geq 1$, the set 
\begin{equation*}
	P_N := \{ (x,y) \in X_{B}^{2} \colon i(x) \in V_m \cap S(B) , i(y) \in V_n \cap S(B) , m,n \leq N , x_k = y_k , \forall k > N \} .
\end{equation*}
Equipped with the relative topology, $P_N$ is a compact principal ample Hausdorff groupoid whose unit space is identified with 
\begin{equation*}
	\bigsqcup_{n = 0}^{N} \bigsqcup_{v \in S(B) \cap V_n} Z(v).
\end{equation*}
We define the groupoid associated with the Bratteli diagram $B$ as the increasing union 
\begin{equation*}
	\G_B := \bigcup_{N = 1}^{\infty} P_N ,
\end{equation*}
equipped with the inductive limit topology. A compact open basis for this topology is given by the cylinder sets $Z(\mu, \lambda)$ corresponding to finite paths $\mu$ and $\lambda$ such that $i(\mu) \in S(B) \cap V_n$, $i(\lambda) \in S(B) \cap V_m$, for some $m,n \in \N_0$, and $t(\mu) = t(\lambda)$. They are defined as 
\begin{equation*}
	Z(\mu, \lambda) := \{ (x,y) \in Z(\mu) \times Z(\lambda) \colon x_{[n + |\mu|+1,\infty)} = y_{[m + |\lambda|+1, \infty)} \},
\end{equation*}
where, for example, $x_{[n+|\mu|+1,\infty)} = x_{n + |\mu| + 1} x_{n + |\mu| + 2} \dots $, for the $x_k \in E_k$ comprising the infinite path $x$. The unit space of $\G_B$ is identified with $X_B$. Setting $\mathcal{K}_n = P_n \cup \G_{B}^{(0)}$, we see that $\G_B$ becomes an AF-groupoid as defined in the beginning of this paragraph.

\begin{proposition} \label{prop: AF groupoids with compact unit space have polynomial growth}
	Any AF-groupoid can be equipped with a continuous length function for which it has polynomial growth.
	\begin{proof}
		Let $\G$ be an AF-groupoid. By \cite[Theorem 3.9]{SkauEtAlAffableEquivalenceRelations}, there is a Bratteli diagram $B = (V,E)$ such that $\G \cong \G_B $. Thus, by Lemma \ref{lem: permanence properties of polynomial growth}, it suffices to show that $\G_B$ has polynomial growth with respect to a continuous length function. We may write $\G_B = \bigsqcup_{n = 1}^{\infty} \mathcal{K}_n$, for clopen principal subgroupoids $\mathcal{K}_n = P_n \cup \G_{B}^{(0)}$ such that $\mathcal{K}_{n} \subset \mathcal{K}_{n+1}$, $\mathcal{K}_{n}^{(0)} = \G_{B}^{(0)}$ and $\mathcal{K}_{n} \setminus \G_{B}^{(0)}$ is compact, defined as in the preceding paragraph. For $A , A^{\prime} \subset V$ two finite subsets, we let $| A E A^{\prime} |$ denote the total number of paths from the vertices in $A$ to the ones in $A^{\prime}$. Recall that $\G_B$ consists of pairs $(x,y) $ such that $i(x) \in V_n \cap S(B)$, $i(y) \in V_m \cap S(B)$, and there is some $N \in \N_0$ such that $m,n \leq N$ and $x_{[N +1,\infty)} = y_{[N+1, \infty)}$. For such $x = x_{n+1} x_{n+2} \dots$ and $y = y_{m+1} y_{m+2} \dots $, it will be convenient to write $x = \bar{e}_1 \dots \bar{e}_n x_{n+1} x_{n+2} \dots$ and $y = \bar{e}_1 \dots \bar{e}_m y_{m+1} y_{m+2} \dots$ where the $\bar{e}_i$ are objects different from any edge in the Bratteli diagram. For any pair $(e,f)$ consisting of edges and/or objects as above, we set $\epsilon_{e,f} := 0$, if $e = f$ and $\epsilon_{e,f} := 1$, if $e \neq f$. Also, for $m \geq 1$, $S_{\leq m}(B)$ will denote the set of all sources in $\bigsqcup_{n = 0}^{m} V_n $. We may suppose without loss of generality that there is a strictly increasing sequence $\{k_i \}_{i \in \N}$ such that $|S_{\leq k_i}(B) E V_{k_i}| < |S_{\leq k_{i+1}}(B) E V_{k_{i+1}}| $, for each $i \in \N $. Otherwise, $\G_{B}^{(0)}$ is finite, and so the length function defined as $x \mapsto 0 $, for all $x \in \G_B^{(0)}$ and $g \mapsto 1 $, for all $g \in \G_B \setminus \G_{B}^{(0)} $, will do. Assume, therefore, that there exists such a subsequence. The map 
		\begin{equation*}
			l (x,y) = \sum_{k = 1}^{\infty} |S_{\leq k}(B) E V_{k}| \epsilon_{x_{k} , y_{k}} ,
		\end{equation*}
		where, for example, $x_k = \bar{e}_{k}$ or $x_k \in E_k$, is a length function on $\G_B$. This is in fact a continuous length function. Indeed, let $g_n \to g$ in $\G_B$. There exist two finite paths $\mu, \lambda$ with $t(\lambda) = t(\mu)$ such that $g \in Z(\mu, \lambda)$.  It follows that there exists $N \in \N$ such that when $n \geq N$, $g_n \in Z(\mu, \lambda)$. Thus, $l(g_n) = l(g)$, for all $n \geq N$. To see why it has polynomial growth, let $R \geq 1$ be given. Let $m \in \N_0$ be the largest integer for which $|S_{\leq m} (B) E V_m | \leq R$. Given $y \in \G_{B}^{(0)}$, if $x \in \G_{B}^{(0)}$ is such that $l(x,y) \leq R$, then $x_k = y_k $, for all $k > m $, and for $k \leq m$, we might have that $x_k \neq y_k $. Thus, 
		\begin{equation*}
			| B_{(\G_{B})_{y}} (R) | \leq | S_{\leq m} (B) E t(y_m) | \leq | S_{\leq m} (B) E V_m | \leq R ,
		\end{equation*}
		showing that $\G_B$ has polynomial growth with respect to $l$.
	\end{proof}
\end{proposition}

\section{Applications} \label{sec: applications of polynomial growth and RDp}
In this section, we will investigate some consequences of polynomial growth and property $RD_p$. There are two main results. The first being that for an étale groupoid $\G$ endowed with a continuous length function for which it has property $RD_p$ and $RD_q$, where $p,q \in (1, \infty)$ are Hölder conjugate, the Fréchet space $S_{p}^{l}(\G)$ becomes a Fréchet algebra under convolution, and we have isomorphisms in $K$-theory 
\begin{equation*}
	K_{\ast}(F_{\lambda}^{p}(\G)) \cong K_{\ast}(S_{p}^{l}(\G)) \cong K_{\ast} (F_{\lambda}^{q}(\G)) ,
\end{equation*}
for $\ast = 0,1$. The second is that when the étale groupoid has polynomial growth, the K-groups of $\Fp(\G)$ are independent of the exponent $p \in (1, \infty)$. Similar results were obtained in \cite{kTheoryOfBAandRD} for locally compact groups. Therein, the analogous result to the second was obtained for a fairly large class of groups which includes groups of polynomial growth. To obtain the analogous result to our first, it was only required that the locally compact groups have property $RD_p$, in which case they found a dense spectral invariant Banach algebra of the reduced group $L^p$-operator algebra. 
The added complexity of multiple units and general lack of symmetry adds difficulties not present in the group case. We are therefore forced to assume both property $RD_p$ and $RD_q$ (or equivalently, by Proposition \ref{prop : RD inherited below (étale case)}, the one corresponding to the larger exponent of the two) and to look for dense spectral invariant Fréchet algebras instead of Banach algebras.

Recall that for an étale groupoid $\G$ endowed with a locally bounded length function $l$, the space of rapidly decreasing functions is 
\begin{equation*}
	S_{p}^{l}(\G) = \bigcap_{k = 0}^{\infty} L_{p,k}(\G) ,
\end{equation*}
and that under the locally convex topology induced by the norms $\{ \lVert \cdot \rVert_{p,k} \}_{k \in \N_0}$, this is a Fréchet space. Elements in $S_{p}^{l}(\G)$ are continuous functions $f$ on $\G$ such that $\lVert f \rVert_{p,k} < \infty$, for all $k \in \N_0$. Recall also that when the étale groupoid $\G$ has property $RD_p$, we can continuously include $S_{p}^{l} (\G)$ into $F_{\lambda}^{p}(\G)$ via $j_{p}^{-1}$. In fact, using Lemma \ref{lem: involution extends to isometric anti-homomorphism} we can also with an analogous argument as in \cite[Lemma 3.3]{hou2017spectral} continuously include $S_{p}^{l} (\G)$ into $F_{\lambda}^{q}(\G)$ via $j_{q}^{-1}$. Moreover, for $f \in S_{p}^{l}(\G)$, $u \in \G^{(0)}$ and $\xi \in C_c(\G_u)$, we have $\lambda_u (f) (\xi) = f \ast \xi $. 

\begin{proposition} \label{prop: Space of rapidly decreasing functions form a Fréchet algebra}
	Let $\G$ be an étale groupoid which has property $RD_p$, for some $p \in (1, \infty)$, with respect to a continuous length function $l \colon \G \to \R_{+}$. Then there exist $c > 0$ and $k \in \N_0$ such that for all $a \in S_{p}^{l}(\G)$, we have 
	\begin{equation*}
		\lVert a \rVert_{F_{\lambda}^{i}(\G)} \leq c \lVert a \rVert_{p,k} ,
	\end{equation*}
	for $i = p,q$. Also, $S_{p}^{l}(\G)$ is a Fréchet $\ast$-algebra with respect to convolution and involution given respectively by 
	\begin{equation*}
		f \ast g (x) = \sum_{y \in \G_{s(x)}} f(x y^{-1}) g(y) ,
	\end{equation*}
	and $f^{\ast}(x) = \overline{f(x^{-1})}$. Setting $\A_p := j_{p}^{-1}(S_{p}^{l}(\G)) $, $\A_q := j_{q}^{-1}(S_{p}^{l}(\G))$ and endowing these with the locally convex topology generated by the norms $\{ \lVert \cdot \rVert_k \}_{k \in \N_0}$ given by $\lVert a \rVert_{k} := \lVert j_p (a) \rVert_{p,k}$ if $a \in \A_p$ and $\lVert a \rVert_k := \lVert j_q (a) \rVert_{p,k}$ if $a \in A_q $, both $\A_p \subset \Fp (\G)$ and $\A_q \subset F_{\lambda}^{q}(\G)$ are dense Fréchet subalgebras which are isomorphic to $S_{p}^{l}(\G)$. Moreover, for $a \in \Fp(\G)$, $a \in \A_p$ if and only if $j_p (a) \in S_{p}^{l}(\G)$, and similarly, for $a \in F_{\lambda}^{q}(\G)$, $a \in \A_q$ if and only if $j_q (a) \in S_{p}^{l}(\G)$.
	
	If $p = 1$, then with respect to the same convolution and involution, $S_{1}^{l}(\G)$ is also a Fréchet $\ast$-algebra, which may be identified with a Fréchet subalgebra of $F_{\lambda}^{1}(\G)$.
	\begin{proof}
		Assume $p \in (1, \infty)$. As already mentioned in the previous paragraph, we know that when $\G$ has property $RD_p$, $S_{p}^{l}(\G)$ is continuously included in $F_{\lambda}^{i}(\G)$, for $i = q,p$. So the first statement regarding the inequality of the norms then follows, say with constants $k \in \N_0$ and $c > 0$. We will show that $S_{p}^{l}(\G)$ is a Fréchet $\ast$-algebra under the above stated convolution and involution. Note first that $\lVert f^{\ast} \rVert_{p,n} = \lVert f \rVert_{p,n}$, for all $n \in \N_0$ and all $f \in S_{p}^{l}(\G)$. Thus, the involution is well defined and continuous. To see that the convolution product is continuous, assume first 
		that $f, g \in S_{p}^{l}(\G)$ are positive functions. Fix $n \in \N_0 $. It is easy to see that 
		\begin{equation*}
			(f \ast g) (x) (1 + l(x))^n \leq (f (1 + l)^n) \ast (g (1+l)^n) (x) ,
		\end{equation*}
		for any $x \in \G$. So, for any $u \in \G^{(0)}$, 
		\begin{align*}
			\big\lVert (f \ast g) (1 + l)^n \big\rVert_{\ell^p (\G_u)} &\leq \big\lVert (f (1 + l)^n) \ast (g (1+l)^n) \big\rVert_{\ell^p (\G_u)} \\
			&\leq \big\lVert f (1+l)^n \big\rVert_{\Fp(\G)} \big\lVert g (1+l)^n \big\rVert_{\ell^p (\G_u)} \leq c \lVert f \rVert_{p, n+k} \lVert g \rVert_{p,n}.
		\end{align*}
		Also, we have
		\begin{align*}
			\big\lVert (f \ast g) (1 + l)^n \big\rVert_{\ell^p (\G^{u})} &\leq \big\lVert (f (1 + l)^n) \ast (g (1+l)^n) \big\rVert_{\ell^p (\G^{u})} \\
			&= \big\lVert (g^{\ast} (1+l)^n) \ast (f^{\ast} (1 + l)^n) \big\rVert_{\ell^p (\G_u)} \\
			&\leq \big\lVert (g(1 + l)^n)^{\ast} \big\rVert_{\Fp(\G)} \lVert f \rVert_{p,n} \\
			&= \big\lVert g (1+l)^n \big\rVert_{\Fq(\G)} \lVert f \rVert_{p,n} \\
			&\leq c \big\lVert g (1+l)^n \big\rVert_{p,k} \lVert f \rVert_{p,n} = c \lVert g \rVert_{p, n +k} \lVert f \rVert_{p,k}.
		\end{align*}
		Taking suprema over all $u \in \G^{(0)}$ and noting that $\lVert \cdot \rVert_{p,t} \leq \lVert \cdot \rVert_{p, t^{\prime}}$, when $t \leq t^{\prime}$, we obtain
		\begin{equation*}
			\lVert f \ast g \rVert_{p,n} \leq c \lVert f \rVert_{p,n+k} \lVert g \rVert_{p,n+k} .
		\end{equation*}
		By the triangle inequality, a similar estimate holds for general $f,g \in  S_{p}^{l}(\G)$. This shows that the product is jointly continuous, and hence $S_{p}^{l}(\G)$ is a Fréchet $\ast$-algebra. By Proposition \ref{prop: identification of Fp functions with C0 functions}, $j_p (ab) = j_p (a) \ast j_p (b)$ for all $a,b \in \Fp(\G)$, from which it follows that $j_{p}^{-1} (f \ast g) = j_{p}^{-1}(f) j_{p}^{-1}(g) $, for all $f,g \in S_{p}^{l}(\G)$. Thus, $j_{p}^{-1}$ is an injective algebra homomorphism and so we may endow its image $\A_p$ with a Fréchet algebra structure for which the statements in the proposition follows. The analogous reasoning shows the same for $\A_q$.
	\end{proof}
\end{proposition}

We record the following lemma for future applications. The proof is straightforward, so we omit it.
\begin{lemma}\label{lem: inclusion of Frechet spaces Slp and equality for RDp, p > 2}
	Let $\G$ be an étale groupoid endowed with a locally bounded length function $l$. If $1 \leq q \leq p < \infty$, then $\lVert \cdot \rVert_{p,k} \leq \lVert \cdot \rVert_{q,k}$, for all $k \in \N_0$, and consequently $ S_{q}^{l}(\G) \subset S_{p}^{l}(\G)$ continuously. If $l$ is continuous and $\G$ has property $RD_p$, for $p > 2$, with respect to $l$, then $ S_{q}^{l}(\G) = S_{p}^{l}(\G)$ as Fréchet algebras, where $q$ is the Hölder exponent of $p$. 
\end{lemma}

In proving our main results, we shall make use of the following series of lemmas. The first one is the Banach algebra version of 
\cite[Theorem 1.2]{JiSmoothDenseSubalgs}. While this is stated in a $C^*$-algebraic setting, the proof carries over directly to Banach algebras. 
\begin{lemma}\label{lem: dense spectral invariant Frechet algebra, Ji} { \cite[Theorem 1.2]{JiSmoothDenseSubalgs}} 
	Let $B$ be a Banach algebra, let $A$ be a Banach subalgebra, and suppose that $\delta \colon \mathrm{Dom}(\delta) \to B$ is a closed unbounded derivation. Let 
	\begin{equation*}
		A_0 := \Big( \bigcap_{k = 0}^{\infty} \mathrm{Dom} (\delta^k) \Big) \cap A .
	\end{equation*}
	Then $A_0$ is a Fréchet algebra under the locally convex topology generated by the semi-norms $\{ \lVert \cdot \rVert_{k} \}_{k \in \N_0}$, where $\lVert a \rVert_{k} := \lVert \delta^k (a) \rVert $, for $a \in A_0$ and $k \in \N_0$. If $A_0$ is dense in $A$, then it is a spectral invariant subalgebra of $A$.
\end{lemma}

If an étale groupoid $\G$ is equipped with a locally bounded length function, we can define, for each $u \in \G^{(0)}$, the set $ \C^p [\G_u]$ consisting of $T \in B(\ell^p (\G_u))$ which satisfy the following three conditions. Firstly, $T$ has finite propagation; that is, 
\begin{equation*}
	\text{Prop}(T) := \sup \{ l(x y^{-1}) \colon T_{xy} \neq 0 \} < \infty ,
\end{equation*}
where, denoting by $\epsilon_y$ the basis vectors corresponding to $y \in \G_u$, $T_{xy} = T(\epsilon_y)(x) $, for $x \in \G_u$. We put $\text{Prop}(0) := 0$. Secondly, $T$ is a translation operator; that is, for each $y \in \G_u$, 
\begin{equation*}
	T (\epsilon_y) = \sum_{z \in F_{y,T}} \alpha_z \epsilon_z = \sum_{z \in F_{y,T}} T_{zy} \epsilon_z ,
\end{equation*}
where $F_{y,T} \subset \G_u$ is a finite subset depending on $y$ and $T$. Thirdly, the absolute value $ |T| $, defined on the basis vectors by 
\begin{equation*}
	|T| (\epsilon_y) = \sum_{z \in F_{y,T}} | \alpha_z | \epsilon_z ,
\end{equation*}
is again in $B(\ell^p (\G_u))$. 

\begin{lemma} \label{lem: translation set is an algebra}
	Suppose $\G$ is an étale groupoid endowed with a locally bounded length function $l$, and let $p \in [1, \infty)$. The set $\C^p [\G_u]$ contains $\lambda_u (C_c (\G))$, and is a subalgebra of $B(\ell^p (\G_u))$.
	\begin{proof}
		We first prove the inclusion $\lambda_u (C_c (\G)) \subset \C^p [\G_u]$. Let $f \in C_c (\G)$ be given. It is easy to see that $\lambda_u (f) (\epsilon_y) = \sum_{z \in \G_{r(y)}} f(z) \epsilon_{zy}$, where the sum is finite because $f$ has compact support. It also follows from this that $\lambda_u (|f |) = | \lambda_u (f) |$, hence $| \lambda_u (f) |$ is again a bounded linear operator on $\ell^p (\G_u)$.  Moreover, since $l$ is locally bounded, 
		\begin{equation*}
			\sup \{ l(x y^{-1}) \colon (\lambda_u (f))_{xy} = f(x y^{-1}) \neq 0 \text{ and } x,y \in \G_u \} \leq \sup\{l(z) \colon z \in \supp (f)\} < \infty .
		\end{equation*}
		Thus, $\lambda_u (C_c (\G)) \subset \C^p [\G_u]$.
		
		Next, we show that $\C^p [\G_u]$ is a linear subspace. Fix any $T,S \in \C^p [\G_u]$ and $\mu \in \C$. We have that $\mu T (\epsilon_y) = \sum_{z \in F_{y,T}} \mu \alpha_z \epsilon_z $, $\text{Prop}(\mu T) \leq \text{Prop}(T)$ (equality if $\mu \neq 0$) and $| \mu T | = | \mu | |T| \in B(\ell^p (\G_u))$, showing that $\mu T \in \C^p [\G_u]$ again. It follows from the inclusions of sets 
		\begin{align*}
			\{ (x,y) \in \G_{u}^{2} \colon (T + S)_{xy} \neq 0 \} &\subset \{ (x,y) \in \G_{u}^{2} \colon T_{xy} \neq 0 \text{ or } S_{xy} \neq 0 \} \\
			&= \{ (x,y) \in \G_{u}^{2} \colon T_{xy} \neq 0 \} \cup \{ (x,y) \in \G_{u}^{2} \colon S_{xy} \neq 0 \},
		\end{align*} 
		that $\text{Prop}(T + S) \leq \text{Prop}(T) + \text{Prop}(S)$. Moreover, we have that 
		\begin{equation*}
			(T + S) (\epsilon_y) = T(\epsilon_y) + S(\epsilon_y) = \sum_{z \in F_{y,T}} \alpha_z \epsilon_z + \sum_{w \in F_{y,S}} \beta_w \epsilon_w = \sum_{x \in F_{y,T} \cup F_{y,S}} (\alpha_x + \beta_x) \epsilon_x ,
		\end{equation*}
		where, for example, $\alpha_x$ is zero if $x \notin F_{y,T} $. In particular, 
		\begin{equation*}
			| T + S | (\epsilon_y) = \sum_{x \in F_{y,T} \cup F_{y,S}} |\alpha_x + \beta_x| \epsilon_x \leq \sum_{z \in F_{y,T}} |\alpha_z | \epsilon_z + \sum_{w \in F_{y,S}} |\beta_w| \epsilon_w = (|T| + |S|)(\epsilon_y).
		\end{equation*}
		Thus, for any positive $\xi \in \ell^p (\G_u)$, 
		\begin{equation*}
			\lVert |T + S | (\xi) \rVert_{\ell^p (\G_u)} \leq \lVert (|T| + |S|) (\xi) \rVert_{\ell^p (\G_u)} ,
		\end{equation*}
		from which it follows that $|T + S|$ is also a bounded linear operator. All of this gives us that $T + S \in \C^p [\G_u]$ again, showing that $\C^p [\G_u]$ is a linear subspace of $B(\ell^p (\G_u))$. 
		
		To see that it is a subalgebra, fix any $T,S \in \C^p [\G_u]$. If $T (\epsilon_y) = \sum_{z \in F_{y,T}} \alpha_z \epsilon_z $ and $S(\epsilon_z) = \sum_{w \in F_{z,S}} \beta_w \epsilon_w$, then 
		\begin{equation*}
			S T (\epsilon_y) = \sum_{z \in F_{y,T}} \sum_{w \in F_{z,S}} \beta_w \alpha_z \epsilon_w = \sum_{z \in F_{y, ST}} \eta_z \epsilon_z.
		\end{equation*}
		It follows that 
		\begin{equation*}
			|S T | (\epsilon_y) = \sum_{z \in F_{y, ST}} |\eta_z | \epsilon_z \leq \sum_{z \in F_{y,T}} \sum_{w \in F_{z,S}} |\beta_w | |\alpha_z | \epsilon_w = |S | |T| (\epsilon_y),
		\end{equation*}
		and so again, given a positive $\xi \in \ell^p (\G_u)$, 
		\begin{equation*}
			\lVert |ST| (\xi) \rVert_{\ell^p (\G_u)} \leq \lVert |S| |T| (\xi) \rVert_{\ell^p (\G_u)},
		\end{equation*}
		so that $|ST| \in B(\ell^p (\G_u))$. Finally, we claim that $\text{Prop}(ST) \leq \text{Prop}(S) + \text{Prop}(T)$. Indeed, notice that 
		\begin{equation*}
			(ST)_{xy} = S (T (\epsilon_y)) (x) = S (\sum_{z \in F_{y,T}} \alpha_z \epsilon_z) (x) = \sum_{z \in F_{y,T}} \alpha_z S(\epsilon_z) (x) .
		\end{equation*}
		Since $S(\epsilon_z) (x) = S_{xz}$ and $\alpha_z = T(\epsilon_y)(z) = T_{zy}$, it follows that 
		\begin{equation*}
			(ST)_{xy} = \sum_{z \in F_{y,T}} S_{xz} T_{zy},
		\end{equation*}
		so if $(ST)_{xy} \neq 0 $, we must have that $T_{zy} \neq 0$ and $S_{xz} \neq 0$, for some $z \in F_{y,T}$. Since $l(x y^{-1}) \leq l(xz^{-1}) + l(z y^{-1})$ and
		\begin{align*}
			\{ (x,y) \in \G_{u}^{2} \colon ST_{xy} \neq 0 \} \subset \{ (x,y) \in \G_{u}^{2} \colon \text{there exists } z \in \G_u \text{ for which } S_{xz} \neq 0 \text{ and } T_{zy} \neq 0 \},
		\end{align*} 
		we have 
		\begin{align*}
			\text{Prop}(ST) =&  \sup \{ l(xy^{-1}) \colon ST_{xy} \neq 0 \} \\
			\leq & \sup \{ l(xy^{-1}) \colon \text{there exists } z \in \G_u \text{ for which } S_{xz} \neq 0 \text{ and } T_{zy} \neq 0 \} \\
			\leq &\sup \{ l(xz^{-1}) + l(z y^{-1}) \colon \text{for } z \in \G_u \text{ such that } S_{xz} \neq 0 \text{ and } T_{zy} \neq 0 \} \\
			\leq &\sup\{ l(xz^{-1}) \colon S_{xz} \neq 0 \} + \sup\{ l(zy^{-1}) \colon T_{zy} \neq 0 \} = \text{Prop}(S) + \text{Prop}(T).				
		\end{align*}
		This shows that $ST \in \C^p [\G_u]$ again, and so $\C^p [\G_u]$ is a subalgebra of $B(\ell^p (\G_u))$.
	\end{proof}
\end{lemma}

The last lemma we shall need is one inspired by \cite[Proposition 4.7]{kTheoryOfBAandRD}. We will denote by $B^p [\G_u]$ the closure of $\C^p [\G_u]$ in $B(\ell^p (\G_u))$. Given a continuous length function $l$, by $M_l$ we shall mean the closed densely defined multiplication operator corresponding to $l$; that is, $M_l (\xi) (x) = l(x) \xi(x)$, for $\xi \in \mathrm{Dom}(M_l) \subset \ell^p (\G_u)$.

\begin{lemma}\label{lem: derivation at u is closed}
	Let $\G$ be an étale groupoid equipped with a continuous length function $l$, fix $p \in [1, \infty)$ and $u \in \G^{(0)}$.  The derivation 
	\begin{equation*}
		\delta_u \colon \mathrm{Dom}(\delta_u) \to B^p [\G_u] , b \mapsto [M_l , b] = M_l b - b M_l ,
	\end{equation*}
	is a closed unbounded derivation. 
	\begin{proof}
		Like in \cite[Proposition 4.7]{kTheoryOfBAandRD}, we shall show that $\delta_u$ may be realized as the infinitesimal generator corresponding to a strongly continuous one parameter group of automorphisms on $B^p [\G_u]$. Notice that for any $t \in \R$, the multiplication operator $e^{i t l}$ is an isometric isomorphism. Since $e^{i t l} (\epsilon_y) = e^{i t l(y)} \epsilon_y $, $\text{Prop}(e^{i t l}) = 0$ and $| e^{i t l} | = I$, we can conclude that $e^{i t l} \in \C^p [\G_u]$. Thus, the one-parameter group of automorphisms 
		\begin{equation*}
			a_t \colon B^p [\G_u] \to B^p [\G_u] , b \mapsto e^{i t l} b e^{-i t l} ,
		\end{equation*}
		is well-defined. If the map $t \mapsto a_t$ is strongly continuous on $B^p [\G_u]$, then the corresponding infinitesimal generator, given by 
		\begin{equation*}
			\delta_a (b) = \lim_{t \to 0} \frac{a_t (b) - b}{t} ,
		\end{equation*}
		is a closed unbounded derivation. Moreover, in this case, adjusting the set-up in \cite[Proposition 2.2]{InfinitesimalGeneratorsOfFlows}, the same proof goes through to the effect that $\delta_a = i \delta_u $, and consequently $\delta_u$ is closed. So we only need to show that $a_t$ is strongly continuous. By density, it suffices to show that given $T \in \C^p [\G_u]$, the map $t \mapsto a_t (T)$ is continuous. Moreover, it suffices to show continuity at zero. To this end, fix any $\xi \in \ell^p (\G_u)$. We may write $\xi = \sum_{y \in \G_u} \alpha_y \epsilon_y$, where $\alpha_y \in \C $. Using that $| e^{i t} -1 | \leq |t| $, for any $t \in \R$, we see that
		\begin{align*}
			\lVert a_t (T) (\xi) - T(\xi) \rVert_{\ell^p (\G_u)}^{p} & = \sum_{x \in \G_u} \Big| e^{it l(x)} \sum_{y \in \G_u} e^{-it l(y)} \alpha_y T(\epsilon_y)(x) - \sum_{y \in \G_u} \alpha_y T(\epsilon_y)(x) \Big|^p \\
			& = \sum_{x \in \G_u} \Big| \sum_{y \in \G_u} (e^{it (l(x) - l(y))} - 1) T_{xy} \alpha_y \Big|^p \\
			& \leq \sum_{x \in \G_u} \Big( \sum_{y \in \G_u}  |t| |l(x) - l(y)| | T_{xy} | | \alpha_y | \Big)^p \\
			& \leq \sum_{x \in \G_u} \Big( \sum_{y \in \G_u} |t| l(xy^{-1}) |T_{xy} | | \alpha_y | \Big)^p \\
			& \leq \text{Prop}(T)^p |t|^p \sum_{x \in \G_u} \Big( \sum_{y \in \G_u} |T_{xy}| |\alpha_y | \Big)^p \\
			& = \text{Prop}(T)^p |t|^p \lVert |T| (|\xi|) \rVert_{\ell^p (\G_u)}^{p} \\
			& \leq \text{Prop}(T)^p |t|^p \lVert | T | \rVert_{B(\ell^p(\G_u))}^{p} \lVert \xi \rVert_{\ell^p (\G_u)}^{p}.
		\end{align*}
		This means that 
		\begin{equation*}
			\lVert a_t (T) - T \rVert_{B(\ell^p (\G_u))} \leq |t | \text{Prop}(T) \lVert |T| \rVert_{B(\ell^p (\G_u) )} ,
		\end{equation*} from which continuity at zero follows. 
	\end{proof}
\end{lemma}

Inspired by the approach in \cite[Theorem 4.2]{hou2017spectral} and \cite{JiSmoothDenseSubalgs}, we have the following result.

\begin{theorem} \label{thm: RD_q and RD_p implies isomorphisms in K-theory}
	Let $\G$ be an étale groupoid which has property $RD_p$ and $RD_q$, for $p,q \in (1, \infty)$ Hölder conjugate, with respect to some continuous length function $l$. We may then identify $S_{p}^{l}(\G)$ as a Fréchet subalgebra of $\Fp(\G)$ and $F_{\lambda}^{q}(\G)$. Under these identifications, $S_{p}^{l}(\G)$ is a spectral invariant Fréchet subalgebra of $F_{\lambda}^{p}(\G) $ and $F_{\lambda}^{q}(\G)$, and the inclusions induce isomorphisms on their $K$-theories. In particular, 
	\begin{equation*}
		K_{\ast} (F_{\lambda}^{p}(\G)) \cong K_* (S^l_p(\G)) \cong K_{\ast}(F_{\lambda}^{q}(\G)).
	\end{equation*}
\end{theorem}
\begin{proof}
	For the sake of exposition, we shall only do the argument in the unital case. Assume without loss of generality that $p \geq q$. Let $u \in \G^{(0)}$ be given, and let $\lambda_{u,p} \colon F_{\lambda}^{p}(\G) \to B(\ell^p (\G_u))$ be the corresponding left regular representation, which for $f \in S_{p}^{l} (\G)$ is given by $\lambda_{u,p}  (f) (\xi)(x) = f \ast \xi (x)$, for $x \in \G_u$ and $\xi \in C_c (\G_u)$. Let $A_{u,p}$ denote the Banach algebra given by the closure of $\lambda_{u,p}  (C_c (\G_u))$ in $B(\ell^p(\G_u))$, and let $\delta_{u,p}$ be the closed derivation on $B^p [\G_u]$ from Lemma \ref{lem: derivation at u is closed}. Notice that $\lambda_{u,p} (\Fp(\G)) \subset A_{u,p} \subset B^p [\G_u]$. For $f \in C_c (\G)$ and $\xi \in \ell^p (\G_u)$, we have that 
	\begin{equation*}
		\delta_{u,p} (\lambda_{u,p}  (f)) (\xi) (x) = \sum_{y \in \G_u} f(x y^{-1}) \xi(y) ( l(x) - l(y) ) ,
	\end{equation*}
	and one can show by induction that 
	\begin{equation*}
		\delta_{u,p}^{k} (\lambda_{u,p} (f)) (\xi) (x) = \sum_{y \in \G_u} f(x y^{-1}) \xi(y) ( l(x) - l(y) )^k ,
	\end{equation*}
	for $k \in \N_0 $. This implies that 
	\begin{align*}
		\big\lVert \delta_{u,p}^{k}(\lambda_{u,p} (f)) \xi \big\rVert_{\ell^{p}(\G_u)}^{p} &= \sum_{x \in \G_u} \big| \delta_{u}^{k} (\lambda_{u,p} (f)) (\xi) (x) \big|^{p} \\
		&= \sum_{x \in \G_u} \Big| \sum_{y \in \G_u} f(x y^{-1}) \xi(y) ( l(x) - l(y) )^k \Big|^{p} \\ 
		&\leq \sum_{x \in \G_u} \Big( \sum_{y \in \G_u} | f(x y^{-1}) | l(x y^{-1})^k | \xi(y) | \Big)^{p} \\
		&\leq \sum_{x \in \G_u} \Big( \sum_{y \in \G_u} f^{(k)}(x y^{-1}) | \xi(y) | \Big)^{p} \\ 
		&\leq \big\lVert \lambda_{u,p}  (f^{(k)}) ( | \xi | ) \big\rVert_{\ell^p (\G_u)}^{p} \leq \big\lVert \lambda_{u,p}  (f^{(k)}) \big\rVert_{B(\ell^p (\G_u))}^{p} \lVert \xi \rVert_{\ell^p (\G_u)}^{p},
	\end{align*}
	where $f^{(k)} (x) = |f (x)| (1 + l(x))^k$. Clearly $f^{(k)} \in C_c(\G)$, and so 
	\begin{equation} \label{eq: norm estimate for iterated derivation}
		\big\lVert \delta_{u,p}^{k} (\lambda_{u,p} (f)) \big\rVert_{B(\ell^p (\G_u))} \leq \big\lVert \lambda_{u,p}  (f^{(k)}) \big\rVert_{B(\ell^p (\G_u))} \leq \big\lVert f^{(k)} \big\rVert_{F_{\lambda}^{p}(\G)} \leq c \lVert f \rVert_{p, k + k^{\prime}},	
	\end{equation}
	where $k^{\prime} \in \N_0$ and $c > 0 $ are the constants given by property $RD_p$. Given $f \in S_{p}^{l}(\G)$, there exists a sequence $f_n \in C_c (\G)$ such that $f_n \to f$ under the norm $\lVert \cdot \rVert_{p, k + k^{\prime}}$. Then $f_n \to f$ in $\Fp (\G)$, so that $\lambda_{u,p}(f_n) \to \lambda_{u,p}(f)$. Also, \eqref{eq: norm estimate for iterated derivation} shows that 
	\begin{equation*}
		\big\lVert \delta_{u,p}^{k} (\lambda_{u,p} (f_n)) - \delta_{u,p}^{k} (\lambda_{u,p} (f_m)) \big\rVert_{B(\ell^p (\G_u))} \leq c \lVert f_n - f_m \rVert_{p,k + k^{\prime}} ,
	\end{equation*}
	and so there exists $a \in B^p [\G_u]$ such that $\delta_{u,p}^{k} (\lambda_{u,p} (f_n)) \to a$. Since $\delta_{u,p}^{k}$ is closed, $\lambda_{u,p}(f) \in \mathrm{Dom}(\delta_{u,p}^{k})$ and $\delta_{u,p}^{k} (\lambda_{u,p} (f)) = a$. Putting 
	\begin{equation*}
		S_{u,p}(\G) = \Big( \bigcap_{k = 0}^{\infty} \mathrm{Dom}(\delta_{u,p}^{k}) \Big) \cap A_{u,p} ,
	\end{equation*}
	the above shows that $S_{u,p} (\G)$ contains $\lambda_{u,p} (S_{p}^{l}(\G))$, and so by Lemma \ref{lem: dense spectral invariant Frechet algebra, Ji}, it is dense and hence spectral invariant in $A_{u,p} $. We define $S_{u,q} (\G)$ similarly as $S_{u,p}(\G)$ with the derivation $\delta_{u,q}$ having the same definition on $B^q [\G_u]$ as $\delta_{u,p}$ had on $B^p [\G_u]$. By an analogous argument as above, one also obtains that $\lambda_{u,q} (S_{p}^{l} (\G) ) \subset S_{u,q}(\G)$, and so again by Lemma \ref{lem: dense spectral invariant Frechet algebra, Ji}, $S_{u,q}(\G)$ is spectral invariant in $A_{u,q}$.

	We know that for $a \in F_{\lambda}^{p}(\G)$, $a^{\ast} \in F_{\lambda}^{q}(\G)$. By Lemma \ref{lem: involution extends to isometric anti-homomorphism}, using the identification provided by the contractive injections $j_p \colon F_{\lambda}^{p}(\G) \to C_0 (\G)$ and $j_q \colon F_{\lambda}^{q}(\G) \to C_0 (\G)$ from Proposition \ref{prop: identification of Fp functions with C0 functions}, $a^{\ast}$ is the function on $\G$ given by $a^{\ast} (x) = \overline{a(x^{-1})}$. If $a \in F_{\lambda}^{p}(\G)$ is such that $\lambda_{u,p} (a) \in S_{u,p} (\G)$ and $\lambda_{u,q} (a^{\ast}) \in S_{u,q} (\G)$, then $\lambda_{u,p} (a) \in \mathrm{Dom}(\delta_{u,p}^{k}) \subset B(\ell^p (\G_u))$ and $\lambda_{u,q} (a^{\ast}) \in \mathrm{Dom}(\delta_{u,q}^{k}) \subset B(\ell^q (\G_u))$, for all $k \geq 0$. If $\epsilon_u$ denotes the basis vector corresponding to the unit $u$, then  
	\begin{equation*}
		\sum_{x \in \G_u} | a(x) |^p l(x)^{pk} = \big\lVert \delta_{u,p}^{k} (\lambda_{u,p} (a)) \epsilon_u \big\rVert_{\ell^p (\G_u)}^{p} \leq \big\lVert \delta_{u,p}^{k} (\lambda_{u,p} (a)) \big\rVert_{B(\ell^p(\G_u))}^p < \infty ,
	\end{equation*}
	and 
	\begin{equation*}
		\sum_{x \in \G^{u}} | a(x) |^q l(x)^{qk} = \big\lVert \delta_{u,q}^{k} (\lambda_{u,q} (a^{\ast})) \epsilon_u \big\rVert_{\ell^q (\G_u)}^{q} \leq \big\lVert \delta_{u,q}^{k} (\lambda_{u,q} (a^{\ast})) \big\rVert_{B(\ell^q(\G_u))}^q < \infty .
	\end{equation*}
	Let $a \in S_{p}^{l}(\G)$ and assume that $b \in F_{\lambda}^{p}(\G)$ is the inverse of $a$ in $F_{\lambda}^{p}(\G)$. Then for each $u \in \G^{(0)}$, $\lambda_{u,p} (b)$ is the inverse of $\lambda_{u,p} (a)$ in $A_{u,p}$. Since $\lambda_{u,p} (a) \in S_{u,p}(\G)$ and $S_{u,p}(\G)$ is spectral invariant, $\lambda_{u,p} (b) \in S_{u,p} (\G)$. Thus, for each $k \in \N_0$, we have 
	\begin{equation*}
		\sum_{x \in \G_u} | b(x) |^p l(x)^{pk} = \big\lVert \delta_{u,p}^{k} (\lambda_{u,p} (b)) \epsilon_u \big\rVert_{\ell^p (\G_u)}^{p} \leq \big\lVert \delta_{u,p}^{k} (\lambda_{u,p} (b)) \big\rVert_{B(\ell^p (\G_u))}^{p} < \infty .
	\end{equation*}
	Also, $b^{\ast}$ is the inverse of $a^{\ast}$ in $F_{\lambda}^{q}(\G)$, and so for each $u \in \G^{(0)}$, $\lambda_{u,q} (b^{\ast})$ is the inverse of $\lambda_{u,q} (a^{\ast})$ in $A_{u,q}$. Since $a^{\ast} \in S_{p}^{l}(\G)$, we have that $\lambda_{u,q} (a^{\ast}) \in S_{u,q}(\G)$. Since $S_{u,q}(\G)$ is spectral invariant in $A_{u,q}$, we have that $\lambda_{u,q} (b^{\ast}) \in S_{u,q}(\G)$; hence for each $k \in \N_0$, we have 
	\begin{equation*}
		\sum_{x \in \G^{u}} | b(x) |^q l(x)^{qk} = \big\lVert \delta_{u,q}^{k} (\lambda_{u,q} (b^{\ast})) \epsilon_u \big\rVert_{\ell^q(\G_u)}^{q} \leq \big\lVert \delta_{u,q}^{k} (\lambda_{u,q} (b^{\ast})) \big\rVert_{B(\ell^q (\G_u))}^{q} < \infty .
	\end{equation*}
	Let $\mathcal{F}^{u} \subset \G^{u}$ be any finite subset. Since $p \geq q$, we have that 
	\begin{equation*}
		\Big( \sum_{x \in \mathcal{F}^{u}} | b(x) |^p l(x)^{pk} \Big)^{q/p} \leq \sum_{x \in \mathcal{F}^{u}} | b(x) |^q l(x)^{qk} \leq \sum_{x \in \G^{u}} | b(x) |^q l(x)^{qk} < \infty .
	\end{equation*}
	It is to obtain the above estimate that we need property $RD_p$ for $p$ the largest Hölder exponent. This means in particular that for any $k \in \N_0$, we have 
	\begin{equation*}
		\sum_{x \in \G^{u}} | b(x) |^p l(x)^{pk} < \infty .
	\end{equation*}
	This is not yet enough to conclude that $b \in S_{p}^{l}(\G)$, since we need 
	\begin{equation*}
		\max \Big\lbrace \sup_{u \in \G^{(0)}} \sum_{x \in \G_u} | b(x) |^p l(x)^{pk} , \sup_{u \in \G^{(0)}} \sum_{x \in \G^u} | b(x) |^p l(x)^{pk} \Big\rbrace < \infty ,
	\end{equation*}
	for all $k \in \N_0$. The binomial theorem together with the inequality $(1 + l(x))^{pk} \leq 2^{pk -1}(1 + l(x)^{pk})$ then implies the $\lVert \cdot \rVert_{p,k}$-norm estimate we really need. But this is indeed the case, because it can be shown by induction that 
	\begin{equation*}
		\sup_{u \in \G^{(0)}} \big\lVert \delta_{u,p}^{k} (\lambda_{u,p} (b)) \big\rVert_{B(\ell^p (\G_u))} < \infty,
	\end{equation*}
	and 
	\begin{equation*}
		\sup_{u \in \G^{(0)}} \big\lVert \delta_{u,q}^{k} (\lambda_{u,q} (b^{\ast})) \big\rVert_{B(\ell^q (\G_u))} < \infty,
	\end{equation*}
	for all $k \in \N_0$. The induction argument uses that $\delta_{u,p}$ and $\delta_{u,q}$ are derivations, that $\lambda_{u,p} (a)$ is the inverse of $\lambda_{u,p} (b)$ and $\lambda_{u,q} (a^{\ast})$ is the inverse of $\lambda_{u,q} (b^{\ast})$, for all $u \in \G^{(0)} $, that 
	\begin{equation*}
		\delta_{u,p}^{k} (\lambda_{u,p} (a) \lambda_{u,p} (b)) = 0 \text{ and } \delta_{u,q}^{k} (\lambda_{u,q} (a^{\ast}) \lambda_{u,q} (b^{\ast})) = 0,
	\end{equation*}
	for all $k \geq 1$, and the base cases are valid because we know that 
	\begin{equation*}
		\sup_{u \in \G^{(0)}} \lVert \lambda_{u,p} (b) \rVert_{B(\ell^p(\G_u))} = \lVert b \rVert_{F_{\lambda}^{p}(\G)} < \infty , 
	\end{equation*}
	and that 
	\begin{equation*}
		\sup_{u \in \G^{(0)}} \lVert \lambda_{u,q} (b^{\ast}) \rVert_{B(\ell^q(\G_u))} = \lVert b^{\ast} \rVert_{F_{\lambda}^{q}(\G)} < \infty .
	\end{equation*}
	
	Now, if $a \in S_{p}^{l}(\G)$ with inverse $b \in F_{\lambda}^{q}(\G)$ in $F_{\lambda}^{q}(\G)$, then $a^{\ast} \in S_{p}^{l}(\G)$ with inverse $b^{\ast} \in F_{\lambda}^{p}(\G)$ in $F_{\lambda}^{p}(\G)$. By what we have done above, $b^{\ast} \in S_{p}^{l}(\G)$, but then $b \in S_{p}^{l}(\G)$. It follows by Lemma \ref{lem: Connes lemma that spectral invariance induce isomorphism in K-theory} that 
	\begin{equation*}
		K_{\ast} (\Fp (\G)) \cong K_\ast (S_{p}^{l}(\G)) \cong K_\ast (F_{\lambda}^{q}(\G)) .
	\end{equation*}
\end{proof}

When the groupoid has polynomial growth, we can say even more about the $K$-theory of the reduced $L^p$-operator algebras associated to it.

\begin{theorem} \label{thm: Polynomial growth implies all K-groups are isomorphic}
	Let $\G$ be an étale groupoid endowed with a continuous length function for which it has polynomial growth. Then the groups $ K_{\ast}(F_{\lambda}^{p}(\G))$, for $p \in (1,\infty)$, are all isomorphic.
	\begin{proof}
		By Proposition \ref{prop: polynomial growth implies RD_q for all q}, $\G$ has property $RD_p$ for all $p \in [1, \infty)$; hence we may apply Theorem \ref{thm: RD_q and RD_p implies isomorphisms in K-theory} to conclude that for all $p \in (1, \infty)$, 
		\begin{equation*}
			K_{\ast}(F_{\lambda}^{p}(\G)) \cong K_{\ast}(S_{p}^{l}(\G)) .
		\end{equation*}
		So if we can show that for all $p \in (1, \infty)$, we have $S_{p}^{l}(\G) = S_{1}^{l}(\G)$ as Fréchet algebras, then the result would follow. First of all, one can apply Hölders inequality together with the fact that $\G$ has polynomial growth to get the continuous inclusion $ S_{p}^{l}(\G) \subset S_{1}^{l}(\G)$; to see this, let $c > 0$ and $k \in \N_0$ be the constant from Proposition \ref{prop: polynomial growth implies RD_q for all q} such that 
		\begin{equation*}
			\sup_{u \in \G^{(0)}} \sum_{x \in \G_u} (1 + l(x))^{-qk} \leq c .
		\end{equation*}
		Then for any $f \in S_{p}^{l}(\G)$, $n \in \N_0$ and $u \in \G^{(0)}$, we have 
		\begin{equation*}
			\sum_{x \in \G_u} | f(x) | (1 + l(x))^n = \sum_{x \in \G_u} | f(x) | (1 + l(x))^{k+n} (1 + l(x))^{-k} \leq c^{1/q} \lVert f \rVert_{p,k+n} .
		\end{equation*}
		Similarly, 
		\begin{equation*}
			\sum_{x \in \G^u} | f(x) | (1 + l(x))^n \leq c^{1/q} \lVert f \rVert_{p,k+n} ,
		\end{equation*}
		so that $\lVert f \rVert_{1,n} \leq c^{1/q} \lVert f \rVert_{p,k+n}$. The other inclusion follows by Lemma \ref{lem: inclusion of Frechet spaces Slp and equality for RDp, p > 2}.
	\end{proof}
\end{theorem}

\section{Examples}\label{sec:Examples}

\subsection{Coarse Groupoids}
In this section, we apply the results on polynomial growth and property $RD_p$ from Section \ref{sec: Length Functions Rapid Decay and Polynomial Growth} and Section \ref{sec: applications of polynomial growth and RDp} to groupoids arising from  bounded geometry coarse structures. Since the graph decomposition result Proposition \ref{prop:graph-decomposition} could be of more general interest, we consider bounded geometry coarse spaces before specializing to the extended metric space setting where we can naturally apply the main results of the paper. We refer the reader to \cite{Roe2003} for material on coarse spaces. 

Let $X$ be a set with coarse structure $\E$. For any controlled set $E \in \E$, we will denote by 
\begin{equation*}
	E_x = \{ y\in X \mid (y,x) \in E \} \quad \text{and } \quad E^x = \{ y \in X \mid (x,y) \in E \},
\end{equation*}
for $x \in X$. Moreover, for $E, F \in \E$, we will denote by
\begin{equation*}
	E \circ F = \{ (x,z) \in X \times X \mid \exists y \in X \text{ such that } (x,y)\in E \text{ and } (y,z) \in F  \}.
\end{equation*}
By $E^n$ we will understand the $n$-iterated product $E \circ \cdots \circ E$. 

As in \cite[Definition 3.25]{Roe2003} and \cite{SkandalisTuYu2002}, we will say that the coarse space $(X,\E)$ is \emph{uniformly locally finite} if
\begin{equation*}
	\sup_{x\in X} \max \{ \vert E_x \vert , \vert E^x \vert  \} < \infty. 
\end{equation*}
for every $E \in \E$.  We define the \emph{growth} of $E \in \E$ in the point $x \in X$ to be the function $\mathrm{gr}_{E,x}\colon \N \to \N$ given by
\begin{equation}\label{eq:coarse-space-growth-function}
	\mathrm{gr}_{E,x} \colon n \mapsto \vert (E^n)_x \vert.
\end{equation}
Note that this is the growth function of \cite[pg. 42]{Roe2003} using the diagonal gauge $\Delta= \{(x,x) \mid x \in X\}$.  
For two functions $f,g \colon \N \to \N$ we write $f \preceq g$ if there exist $a,b >0$ such that $f(n) \leq a g(bn)$ for all sufficiently large $n$. We write $f \approx g$ if $f \preceq g$ and $g \preceq f$. We may then make sense of the growth type of a coarse space $X$. We say $X$ has the \emph{growth type} of $f \colon \N \to \N$ if for every $E \in \E$ and every $x \in X$, $\mathrm{gr}_{E,x} \preceq f$, and for at least one $E$ and one $x$, $f \preceq \mathrm{gr}_{E,x}$. In particular, $X$ has \emph{polynomial growth} if this is true for a polynomial $f$.

We proceed to cover the graph decomposition lemma \cite[Lemma 4.10]{Roe2003}. A \emph{partial bijection} of a coarse space $(X, \E)$ is a triple $(D,R, \tau)$ consisting of two subsets $D$ and $R$ of $X$ together with a bijection $\tau \colon D \to R$. If the graph $\Gamma_\tau := \{ (\tau(x), x) \mid x \in D \} \in \E$, we say $(D,R, \tau)$ is a \emph{partial translation}. We say that a collection $\{ \tau_1, \ldots , \tau_n \}$ of partial bijections of $X$ are \emph{orthogonal} if $\Gamma_{\tau_i} \cap (D_j \times R_j) = \emptyset$ whenever $i \neq j$. Note that this is strictly stronger than the graphs being disjoint. The following lemma contains the parts of \cite[Lemma 4.10]{Roe2003} we need for our purposes.
\begin{lemma}[{ \cite[Lemma 4.10]{Roe2003}}]\label{lemma:Roe-graph-decomposition-lemma-410}
	Let $X$ be a set and let $E \subseteq X \times X$. The following conditions are equivalent:
	\begin{enumerate}
		\item[a)] $E$ is the union of the graphs of a finite orthogonal set of partial bijections;
		\item[b)] $E$ is controlled for the universal bounded geometry coarse structure (on $X$); that is, $\sup_{x \in X} \max \{ \vert E_x \vert , \vert E^x \vert \} < \infty$. 
	\end{enumerate}
\end{lemma}
We proceed to explain how one goes from b) to a) in Lemma \ref{lemma:Roe-graph-decomposition-lemma-410}, that is how one constructs the finite number of orthogonal graphs of partial bijections whose (disjoint) union make up $E$. This follows the construction in the proof of \cite[Lemma 4.10]{Roe2003}.

So let $E$ be a controlled set for the universal bounded geometry coarse structure on $X$, meaning $N(E) := \sup_{x \in X} \max \{ \vert E^x \vert , \vert E_x \vert \} < \infty$. Define a graph $\Gamma$ whose vertices are the points of $E$ and such that $(x,y)$ and $(x', y')$ are linked by an edge if and only if either $(x, y') \in E$ or $(x', y)\in E$. One may then use a greedy algorithm for vertex coloring of this graph to decompose $E$ into orthogonal graphs of partial bijections. In short, we may decompose $E$ into at most $n+1$ orthogonal graphs of partial bijections provided the maximal vertex degree of $\Gamma$ is $n$. Next, we show how to bound the vertex degree $n$ in terms of $N(E)$. 

\begin{proposition}\label{prop:graph-decomposition}
	Let $(X, \E)$ be a coarse space and let $E \in \E$ be so that 
	\begin{equation*}
		N :=  \sup_{x \in X} \max \{ \vert E^x \vert , \vert E_x \vert \} < \infty.
	\end{equation*}
	The vertex degree of the graph $\Gamma$ defined above is bounded by $2N(N-1)$. Thus, at most $2N(N-1) +1$ orthogonal graphs of partial bijections are needed in the decomposition of $E$ in Lemma \ref{lemma:Roe-graph-decomposition-lemma-410}. 
\end{proposition}

\begin{proof}
	Let $(x_0, y_0)$ be an arbitrary vertex in $\Gamma$ (that is an arbitrary point in $E$). We wish to bound its vertex degree. By assumption on $E$, $\vert E^{x_0} \vert \leq N$ and $\vert E_{y_0} \vert \leq N$. 
	
	Other than $(x_0,y_0)$, there are at most $N-1$ points in $E$ with first coordinate $x_0$. Say there are $K_{x_0}$ of them, and label them $(x_0, y_1), \ldots , (x_0, y_{K_{x_0}})$.
	
	Likewise, other than $(x_0, y_0)$, there are at most $N-1$ points in $E$ with second coordinate $y_0$. Say there are $K_{y_0}$ of them, and label them $(x_1, y_0 ), \ldots , (x_{K_{y_0}}, y_0)$.
	
	If $(x,y) \neq (x_0, y_0)$ connects to $(x_0, y_0)$, then either
	\begin{itemize}
		\item $(x,y_0) \in E$, which implies $x \in \{x_1, \ldots , x_{K_{y_0}} \}$. For each such choice of $x$, the point $(x,z)$ connects to $(x_0, y_0)$ as $(x,y_0) \in E$. By assumption of uniform local finiteness, there are at most $N$ choices for $z$; or
		\item $(x_0, y) \in E$, which implies $y \in \{y_1, \ldots , y_{K_{x_0}}  \}$. For each such choice of $y$, the point $(z,y)$ connects to $(x_0, y_0)$ as $(x_0, y) \in E$. Again by uniform local finiteness, there are at most $N$ choices for $z$.
	\end{itemize}
	In total, the vertex degree of $(x_0, y_0)$ is bounded by
	\begin{equation*}
		K_{x_0} \cdot N + K_{y_0} \cdot N \leq (N-1)N + (N-1)N = 2N(N-1).
	\end{equation*}
	Thus, by a greedy vertex coloring algorithm, the number of orthogonal graphs needed to decompose $\Gamma$ is bounded by $2N(N-1) + 1$.
\end{proof}

\begin{remark}
	With the given graph coloring algorithm used, 
	we can not hope to get the bound sub-quadratic. Indeed, it is not hard to see that for the bounded coarse structure on $\Z$ induced by the usual metric, we have a quadratic lower bound on the number of orthogonal graphs used in the decomposition of $E_k := \{ (m,n) \in \Z \times \Z  \mid \vert m-n \vert \leq k \}$ around $(0,0)$, say. 
\end{remark}

For a discrete coarse space $X$, denote by $\beta X$ the Stone-\v{C}ech compactification of $X$.

\begin{definition}\label{def:coarse-groupoid}
	Let $X$ be a uniformly locally finite coarse space, and denote the coarse structure by $\E$. The 
	\emph{coarse groupoid} of $X$ (with respect to the coarse structure $\E$) is the subset $\G(X) \subseteq \beta (X \times X)$ defined by
	\begin{equation*}
		\G (X) := \bigcup_{E \in \E} \overline{E}
	\end{equation*}
	where the closures are taken in $\beta (X \times X)$. 
\end{definition}
By \cite[Corollary 10.18]{Roe2003}, we may view closures of controlled sets $E$ as being in either $\beta (X \times X)$ or in $\beta X \times \beta X$ depending on what best fits our purposes. 
Then \cite[Proposition 3.2]{SkandalisTuYu2002} (or \cite[Proposition 10.20]{Roe2003}) tells us that $\G (X)$ indeed becomes a groupoid. The source, target, inverse and multiplication maps on the pair groupoid $X \times X$ have unique continuous extensions to $\G (X)$. With respect to these extensions, $\G (X)$ becomes a principal, \'etale, locally compact and Hausdorff groupoid with unit space $\beta X$. 

By \cite[Example 10.25]{Roe2003} we obtain that for $x \in X$, the fiber $\G(X)_x$ equals the coarse connected component (see \cite[Remark 2.20]{Roe2003}) of $x$ in $X$. Thus, the growth properties of $\G(X)_x$, for $x \in X$, is entirely determined by the growth properties of the original coarse space $(X, \E)$. The key to understanding the growth of the groupoid $\G(X)$ therefore lies in understanding $\G(X)_\omega$ for $\omega \in \beta X \setminus X$. 

\begin{proposition}\label{prop:cardinalities-of-sections-in-coarse-groupoids}
	Let $X$ be a set equipped with a uniformly locally finite coarse structure $\E$, that is, $\sup_{x \in X} \max \{ \vert E_x \vert , \vert E^x \vert \} < \infty$, for all $E \in \E$. 
	Suppose $X$ has growth type bounded by $f \colon \N \to \N$.  
	Then for any $E \in \E$ and $\chi \in \beta X$, we have
	\begin{equation*}
		[n \mapsto \vert (\overline{E}^n)_\chi \vert] \preceq 
		\begin{cases}
			f^2 & \text{ if $\chi \in \beta X \setminus X$} \\
			f & \text{ if $\chi \in X$}
		\end{cases}
	\end{equation*}
\end{proposition}
\begin{proof}
	First suppose $\chi \in X$. Since $\G(X)_\chi$ equals the coarse connected component of $\chi$ in $X$, we get $ (\overline{E}^n)_\chi = (E^n)_\chi$, from which $[n \mapsto \vert (\overline{E}^n)_\chi \vert] \preceq f$ follows by assumption. 
	
	So let $\chi \in \beta X \setminus X$, let $E \in \E$, and let $n \in \N$. First note that by continuity of multiplication in $\G(X)$, we  have $\overline{E}^n \subseteq \overline{E^n}$. So it suffices to show that 	$[n \mapsto \vert (\overline{E^n})_\chi \vert] \preceq f^2$. 
	If $(\omega, \chi) \in  \overline{E^n}$, then $(\omega, \chi) \in \overline{\Gamma_\tau}$ for some graph $\Gamma_\tau$ of a partial translation $\tau$. In other words, $(\omega, \chi) = \lim_i (\tau(x_i), x_i)$ for $(x_i)_i \subseteq E^n$ converging to $\chi$. But Proposition \ref{prop:graph-decomposition} tells us that the number of orthogonal graphs needed to decompose $\overline{E^n}$ is at most of the order of $f(n)^2$. 
	It follows that $[n \mapsto \vert (\overline{E}^n)_\chi \vert] \preceq f^2$ for $\chi \in \beta X \setminus X$. 
\end{proof}

We now specialize to the case of an extended metric space, where there is a natural length function appearing. The \emph{bounded coarse structure} on an extended metric space  $(X,d)$ is the coarse structure whose controlled subsets are subsets of $E_r =\{ (x,y) \in X \times X \mid d(x,y)\leq r  \}$. Recall that an extended metric space $(X,d)$ is said to be \emph{uniformly locally finite} if for any $R >0$ there is a uniform finite upper bound on the cardinalities of all closed balls with radius $R$, that is
\begin{equation*}
	\sup_{x \in X} \vert \overline{B}(x, R) \vert < \infty 
\end{equation*}
where $ \overline{B}(x, R) $ denotes the closed $R$-ball around $x \in X$. This is exactly the condition on the extended metric space that will make the associated bounded coarse structure be uniformly locally bounded. Associated to $(X,d)$ is the \emph{coarse groupoid} $\G_{(X,d)}$ which is defined as follows
\begin{itemize}
	\item for any $r\geq 0$, define the controlled set $E_r := \{ (x,y) \in X \times X \mid d(x,y) \leq r \}$;
	\item as a topological space, we have $\G_{(X,d)} := \bigcup_{r \geq 0} \overline{E_r}$ inside $\beta (X \times X)$, the Stone-\v{C}ech compactification of $X \times X$;
	\item we have $\G^{(0)}_{(X,d)} = \overline{E_0} \cong \beta X$;
	\item the range and source maps are the unique extensions of the first and second factor maps $X \times X \to X$, respectively;
	\item The multiplication is inherited from the pair groupoid multiplication on $\beta X \times \beta X$.
\end{itemize}

In \cite[Remark 5.17]{ma2020almostElem}, Ma and Wu make several claims regarding the length function on $\G_{(X,d)}$, the induced extended metric, and the structure of the source fibers $(\G_{(X,d)})_x$, for $x \in X$. We collect the observations in the following lemma.

\begin{lemma}\label{lemma:coarse-groupoid-metric-length-observations}
	Let $\G_{(X,d)}$ be a coarse groupoid of a uniformly locally finite extended metric space $(X,d)$. Then:
	\begin{enumerate}
		\item the extended metric $d$ extends to an extended metric $\beta d$ which takes finite values on $\G_{(X,d)}$;
		\item the induced length function $\ell_{\beta d}$ is continuous and proper;
		\item for any $x \in X$, the source fiber $(\G_{(X,d)})_x$ equipped with the fiberwise invariant metric induced by $\ell_{\beta d}$ is isometrically isomorphic to $(\mathrm{conn}_x(X),d\vert_{\mathrm{conn}_x(X)})$, where  $\mathrm{conn}_x(X)$ is the (metric space) connected component of $x$ in $X$, and $d\vert_{\mathrm{conn}_x(X)}$ is $d$ restricted to this connected component. 
	\end{enumerate}
\end{lemma}
\begin{proof}
	1) follows easily by viewing the metric as taking values in $[0,\infty]$ with the topology of the Alexandroff compactification of $[0,\infty)$. Moreover, 2) is written out in \cite[Remark 5.17]{ma2020almostElem}.
	
	To prove 3), note that the length function $\ell_{\beta d}$ yields an extended metric $\rho_{\ell_{\beta d}}$ on $\G_{(X,d)}$ by defining
	\begin{equation*}
		\rho_{\ell_{\beta d}} ((x,y), (u,v)) := 
		\begin{cases}
			\ell_{\beta d} (x,u) & \text{if $y=v$} \\
			0 & \text{otherwise}
		\end{cases}
		=
		\begin{cases}
			\beta d(x,u) & \text{if $y=v$} \\
			0 & \text{otherwise}
		\end{cases}
	\end{equation*}
	We note that for $x \in X$, the source fiber $(\G_{(X,d)})_x$ will (as a set) equal $\mathrm{conn}_x(X)$ by \cite[Example 10.25]{Roe2003}. It is then clear that $((\G_{(X,d)})_x, \rho_{\ell_{\beta d}})$ is isometrically isomorphic to $(\mathrm{conn}_x(X),d\vert_{\mathrm{conn}_x(X)})$. 
	
\end{proof}

We make the following observations. Let $(X,d)$ be a uniformly locally finite extended metric space, and let $E_r \in \E$ be a controlled set for the associated bounded coarse structure, as in the definition of the associated coarse groupoid above. Suppose there is a function $f \colon \R_{\geq 0} \to \R$ such that 
\begin{equation*}
	\vert \overline{B}(x,r) \vert \leq f(r),
\end{equation*}
for all $x \in X$ and all $r \in \N$, that is, a growth condition on the cardinalities of $r$-balls in our extended metric space. We see that this implies the bounded coarse structure associated with $(X,d)$ has growth type bounded by that of $f$. 

Estimating the growth of $(\G_{(X,d)})_\chi$ we get, for $\chi \in \beta X$ and a sufficiently large constant $M>0$,
\begin{equation*}
	\begin{split}
		\vert B_{(\G_{(X,d)})_\chi} (r)\vert &= \vert \{ \nu \in \beta X \mid (\nu,\chi) \in \overline{E_r} \} \vert \\
		&\leq \begin{cases}
			f(r) & \text{ if $\chi \in X$} \\
			M f(r)^2 & \text{ if $\chi \in \beta X \setminus X$}
		\end{cases}
	\end{split}
\end{equation*}
This follows from applying Proposition \ref{prop:cardinalities-of-sections-in-coarse-groupoids} with $n=1$. 
Thus, we obtain

\begin{proposition}\label{prop:metric-poly-growth-implies-groupoid-poly-growth}
	Let $(X,d)$ be a uniformly locally finite extended metric space, and suppose there is a function $f \colon \R_{\geq 0} \to \R$ for which $\vert \overline{B}(x,r)\vert \leq f(r)$ for all $x \in X$. Then, 
	\begin{equation*}
		\vert B_{(\G_{(X,d)})_\chi} (r) \vert \leq \begin{cases}
			f(r) & \chi \in X \\
			Mf(r)^2 & \chi \in \beta X \setminus X
		\end{cases}
	\end{equation*}
	for a sufficiently large constant $M$. 
	That is, the growth of the groupoid $\G_{(X,d)}$ is bounded above by the growth type of $f^2$. In particular, if $f$ is a polynomial, then $\G_{(X,d)}$ has polynomial growth. 
\end{proposition}

Combining Proposition \ref{prop:metric-poly-growth-implies-groupoid-poly-growth} with Proposition \ref{prop: polynomial growth implies RD_q for all q} and Theorem \ref{thm: Polynomial growth implies all K-groups are isomorphic} immediately yields

\begin{proposition}\label{prop:implications-for-extended-metric-spaces}
	Let $(X,d)$ be a uniformly locally finite extended metric space for which there is a polynomial $f$ such that $\vert \overline{B}(x,r)\vert \leq f(r)$ for all $x \in X$ and all $r \geq 0$. Denote by $\G_{(X,d)}$ the associated coarse groupoid. Then
	\begin{enumerate}
		\item $\G_{(X,d)}$ has property $RD_p$ for all $p \in (1,\infty)$;
		\item The $K$-theory groups $K_* (F^p_{\lambda} (\G_{(X,d)}))$, for $*=0,1$, are independent of $p \in (1,\infty)$. 
	\end{enumerate}
\end{proposition}

The result \cite[Proposition 10.28]{Roe2003} together with the analogous proof as to \cite[Proposition 10.29]{Roe2003} for $p\in (1,\infty)$ rather than just the case $p=2$, establishes a contractive bijective homomorphism from $\Fp(\G_{(X,d)})$ to $B^p_u (X,d)$. Here $B^p_u (X,d)$ denotes the uniform $L^p$-Roe algebra, that is, the operator norm closure in $B(\ell^p (X))$ of the algebra of all bounded operators of finite propagation. It follows that for any $p \in (1,\infty)$ and $*=0,1$, $K_* (\Fp(\G_{(X,d)})) \cong K_* (B^p_u (X,d))$. 
The following is then immediate by Proposition \ref{prop:implications-for-extended-metric-spaces}.
\begin{corollary}\label{corollary:roe-algebra-k-theory-results}
	Let $(X,d)$ be a uniformly locally finite extended metric space for which there is a polynomial $f$ such that $\vert \overline{B}(x,r)\vert \leq f(r)$ for all $x \in X$ and all $r \geq 0$. Then the $K$-theory groups $K_* (B^p_u (X,d))$,  for $*=0,1$ are independent of $p \in (1,\infty)$. 
\end{corollary}
Though they consider a different algebra, similar results were obtained in \cite[Theorem 5.22]{Zhang&ZhouLp}, where they prove that the $K$-theory groups of the Roe algebras $B^p (X,d)$ are independent of $p \in (1, \infty)$ if $X$ is a proper metric space with finite asymptotic dimension. Note that unlike the results of this paper, they do not require uniform local finiteness. They ask in \cite[Question 6.2]{Zhang&ZhouLp} if it is possible to show such independence of $p$ without using Baum-Connes. Corollary \ref{corollary:roe-algebra-k-theory-results} shows that under restrictions on $(X,d)$ of polynomial growth and uniform local finiteness, one can prove the $K$-theory of the uniform $L^p$-Roe algebras is independent of the value of $p \in (1, \infty)$ by comparatively simple means.

\subsection{Point Set Groupoids}
In this section, we will show that groupoids arising from point sets as in for example \cite{EnstadRaum2022} have polynomial growth under natural assumptions on the locally compact group and the point set. As such, they also lend themselves to applications of the results from Section \ref{sec: Length Functions Rapid Decay and Polynomial Growth} and Section \ref{sec: applications of polynomial growth and RDp}. 

Let $G$ be a locally compact secound-countable group equipped with a continuous and proper length function $l$ under which $G$ has polynomial growth. Denote the Haar measure on $G$ by $m$. Assume further that $\Lambda \subseteq G$ is a separated point set, that is, there exists a unit-neighborhood $U \subseteq G$ such that the cardinality $\vert \Lambda \cap xU \vert \leq 1$ for all $x \in G$, where $xU$ is the translation of $U$ by $x$. 

Denote by $\mathcal{C}(G)$ the set of closed subsets of $G$, and equip it with the Chabauty-Fell topology. It is known that $\mathcal{C}(G)$ is compact with this topology. There is a continuous left $G$-action on $\mathcal{C}(G)$ given by
\begin{equation*}
	xC = \{ xy \mid y\in C \},
\end{equation*}
for $x \in G$ and $C \in \mathcal{C}$. The closure of the orbit of $\Lambda$ under this action is called the hull of $\Lambda$, and we denote it by $\Omega(\Lambda)$. The punctured hull of $\Lambda$ is defined as $\Omega^{\times}(\Lambda) = \Omega (\Lambda) \setminus \emptyset$. One then further defines the transversal of $\Lambda$ to be
\begin{equation*}
	\Omega_0 (\Lambda) := \{ P \in \Omega(\Lambda) \mid e \in P \} \subseteq \Omega^{\times}(\Lambda),
\end{equation*}
where $e$ is the identity of $G$. We remark that $\Omega_0 (\Lambda)$ is compact as it is closed in $\Omega (\Lambda)$. 

We then consider the transformation groupoid $G \ltimes \Omega^{\times}(\Lambda)$. The \emph{point set groupoid} of $\Lambda$, denoted by $\G(\Lambda)$, is the restriction of this transformation groupoid to the transversal $\Omega_0 (\Lambda)$ of $\Lambda$, that is
\begin{equation}\label{eq:groupoid-of-point-set}
	\G(\Lambda) = \{ (x,P) \in G \times \Omega^{\times} (\Lambda)  \mid P, xP \in \Omega_0 (\Lambda) \} = \{(x,P) \in G \times \Omega_0(\Lambda) \mid x^{-1} \in P  \}.
\end{equation}
The operations are the ones inherited from the transformation groupoid. 
It is shown in \cite[Proposition 3.11]{EnstadRaum2022} that when $\Lambda$ is separated, $\G(\Lambda)$ is \'etale. 

\begin{proposition}
	Let $G$ be a locally compact group endowed with a length function $l \colon G \to \R_{\geq 0}$ with respect to which $G$ has polynomial growth. Suppose $\Lambda \subset G$ is a separated point set, and let $\G (\Lambda)$ be the groupoid described in \eqref{eq:groupoid-of-point-set}. Then the map $L\colon \G(\Lambda) \to \R_{\geq 0}$ given by $L(x,P) = l(x)$ defines a length function on $\G(\Lambda)$, under which $\G(\Lambda)$ has polynomial growth. 
\end{proposition}
\begin{proof}
	We first verify that $L$ defines a length function on $\G(\Lambda)$. If $L(x,P) = 0$, then $l(x) = 0$, meaning $x=e$. Thus, $(x,P) = (e,P) \in \G(\Lambda)^{(0)}$. It is also sub-multiplicative, as
	\begin{equation*}
		\begin{split}
			L((y,xP)(x,P)) = L(yx,P) = l(yx) \leq l(y) + l(x) = L(y,xP) + L(x,P)
		\end{split}
	\end{equation*}
	and it is invariant under inversion since
	\begin{equation*}
		L((x,P)^{-1}) = L(x^{-1}, xP) = l(x^{-1}) = l(x) = L(x,P).
	\end{equation*}
	We conclude that $L$ is a length function on $\G(\Lambda)$. 
	
	To see that $\G(\Lambda)$ has polynomial growth, let $P \in \Omega_0(\Lambda) = \G(\Lambda)^{(0)}$. We will use in the sequel that if $(x,P)\in \G(\Lambda)_{P}$, then $x^{-1} \in P$, as described in \eqref{eq:groupoid-of-point-set}. Let $\overline{B}^{l}_r(e)$ denote the closed $l$-ball of radius $r$ in $G$ around the unit. Then, using that $l(x^{-1}) = l(x)$ for all $x \in G$, we get
	\begin{equation*}
		\begin{split}
			\vert B_{\G(\Lambda)_{P}} (r)\vert &= \vert \{ (x,P) \mid l(x) \leq r \} \vert = \vert P \cap \overline{B}^l_r (e) \vert .
		\end{split}
	\end{equation*}
	Using \cite[Corollary 3.4]{EnstadRaum2022}, we see that
	\begin{equation*}
		\vert P \cap \overline{B}^l_r (e) \vert \leq \frac{ m(\overline{B}^l_r (e) V)  }{m (V)} ,
	\end{equation*}
	for any symmetric unit-neighborhood $V$ of $G$. Choosing $V$ to be an $l$-ball in $G$ of some finite radius we see that the denominator is just some constant, while the numerator grows polynomially by assumption of polynomial growth of the group $G$. Since $P \in \Omega_0(\Lambda)$ was arbitrary and this estimate is independent of $P$, we conclude that $\G(\Lambda)$ has polynomial growth.
\end{proof}

\subsection{Graph Groupoids}
In this subsection, we shall characterize the groupoids arising from finite directed graphs that have polynomial growth. A specific instance of such groupoids has as its reduced $L^p $-operator algebra the $L^p $-Toeplitz algebras (see Example \ref{example: graph giving Lp-toeplitz}), and so Theorem \ref{thm: Polynomial growth implies all K-groups are isomorphic} will apply to give the same result as \cite[Theorem 4.3]{WangWang}. Let us start by defining finite directed graphs and their associated groupoids.

By a directed graph, we mean a quadruple $E=(E^0,E^1,s,r)$, where $E^0$ and $E^1$ are non-empty sets called the sets of vertices and edges, respectively, and $r,s \colon E^1 \to E^0$ are maps called the range and source maps. We view $e \in E$ as an arrow from $s(e)$ to $r(e)$. The directed graph $E$ is said to be \emph{finite} if both $E^0$ and $E^1$ are finite. A vertex $v\in E^0$ such that $s^{-1}(v)=\emptyset$ is called a \emph{sink}. We denote by $E^0_{sink}$ the subset of all sinks.  A path of length $n \in \N$ is an $n$-tuple $\alpha=(\alpha_1,\ldots, \alpha_n)\in (E^1)^n$ with $r(\alpha_i)=s(\alpha_{i+1})$, for $i=1,\ldots, n-1$. 
We will denote by $E^n$ all the paths of length $n$, and write $\alpha = \alpha_1 \dots \alpha_n$ instead of $(\alpha_1 , \dots , \alpha_n)$. The vertices are by convention paths of length zero. Then $E^*:=\bigcup_{n=0}^\infty E^n$ is the set of all finite paths. Given $\alpha\in E^*$ we denote by $|\alpha|$ its length, and we define $s(\alpha) : = s(\alpha_1)$ and $r(\alpha) := r(\alpha_{|\alpha|})$ when $\alpha \in E^n$ with $n \geq 1$. For $v \in E^0$ we put $r(v) = s(v) = v$.
The \emph{infinite path space} of $E$ is the set of all sequences $x=(x_n)_{n=1}^\infty$, where $x_i \in E^1$ and $r(x_i)=s(x_{i+1})$, for every $i\in \N$, and it will be denoted by $E^\infty $. The \emph{boundary path space of $E$} is the set
\begin{equation*}
	\partial E:=E^\infty \cup \{\alpha\in E^*: r(\alpha)\in E^0_{sink} \}.
\end{equation*} 
With the topology generated by the sets of the form 
\begin{equation*}
	Z(v) := \{ x \in \partial E \colon s(x) = s(x_1) = v \} ,
\end{equation*}
and
\begin{equation*}
	Z(\alpha):=\{x\in \partial E \colon x_i=\alpha_i\text{ for }1\leq i\leq |\alpha| \} ,
\end{equation*}
for $\alpha\in E^*$ with $| \alpha | \geq 1$, we have that $\partial E$ is a totally disconnected space. Given $\alpha \in E^* \setminus E^0$ and $x\in \partial E$ with $s(x) = r(\alpha)$, we define the concatenated path $\alpha x$ by $(\alpha x)_i=\alpha_i$ for $1\leq i\leq |\alpha|$ and $(\alpha x)_i= x_{i-|\alpha|}$ for $i\geq |\alpha|+1 $. 
The \emph{graph groupoid} of $E$ is defined as 
\begin{equation*}
	\G_E:=\{(\alpha x,|\alpha|-|\beta|,\beta x): \alpha,\beta \in E^*, x\in \partial E \text{ with }r(\alpha)=r(\beta)=s(x)\},
\end{equation*}
where $((\alpha x,|\alpha|-|\beta|,\beta x),(\delta x',|\delta|-|\gamma|,\gamma x'))\in \G_E^{(2)}$ if and only if $\beta x=\delta x'$, and in this case $(\alpha x,|\alpha|-|\beta|,\beta x)\cdot (\delta x',|\delta|-|\gamma|,\gamma x')=(\alpha x,|\alpha|-|\beta|+|\delta|-|\gamma|,\gamma x')$. Moreover, we have that $(\alpha x,|\alpha|-|\beta|,\beta x)^{-1}=(\beta x, |\beta|-|\alpha|,\alpha x)$. With the topology given by the sets
\begin{equation*}
	Z(\alpha,\beta)=\{(\alpha x,|\alpha|-|\beta|, \beta x): x\in \partial E \text{ with }r(\alpha)=r(\beta)=s(x)\} ,
\end{equation*}
for $\alpha,\beta\in E^*$ with $r(\alpha)=r(\beta)$, $\G_E$ is an ample second countable groupoid. The unit space is $\G_E^{(0)}=\{(x,0,x):x\in \partial E\}$, and is identified with $\partial E$. Given $x\in E^{\infty}$, we define $\sigma(x)$ as the infinite path with $\sigma(x)_i=x_{i+1} $, for every $i\in \N$. Given $\alpha \in E^{\ast}$, with $| \alpha | \geq 2 $, say $\alpha = \alpha_1 \dots \alpha_{| \alpha|}$, we define $\sigma(\alpha)$ to be the path $\alpha_2 \dots \alpha_{| \alpha|}$, and if $| \alpha | = 1 $, we define $\sigma (\alpha) = r(\alpha) $. With the above notation introduced, we can describe
\begin{equation} \label{eq: description of fibers in graph groupoids}
	(\G_E)_x=\{ (\alpha \sigma^n(x),|\alpha|-n,x) \colon n\in \N_0, \alpha\in E^* \text{ with }r(\alpha)=s(x_{1+n}) \} .
\end{equation}
Let $E$ be a finite directed graph, with associated graph groupoid $\G_E$. The set 
\begin{equation*}
	S = \bigcup_{\alpha\in E^1}\Big( Z(\alpha,r(\alpha))\cup Z(r(\alpha),\alpha)\Big) ,
\end{equation*}
is compact open, and is a \emph{generating set} for $\G_E $, meaning that $S^{-1} = S $, and for any $g \in \G_E $, there exists $N \in \N$ such that $\bigcup_{k = 1}^{N} S^k $ is a neighborhood of $g$ in $\G_E $. Such a compact open generating set induces a continuous length function, given by $l_S (u) = 0 $, for each $u \in \G_{E}^{(0)}$, and for $g \notin \G_{E}^{(0)}$ by 
\begin{equation} \label{eq: canonical length function for graph groupoids}
	l_S (g) = \inf \{ n \in \N \colon g \in \bigcup_{k = 1}^{n} S^k \}.
\end{equation}

In what follows, we shall refer to $l_S$ as the \emph{canonical length function} of $\G_E$. 

Note that given $g=(\alpha \sigma^n(x),|\alpha|-n,x)\in (\G_E)_x $, we have 
\begin{equation}\label{eq: product of basis elements}
	g\in Z(\alpha_1,r(\alpha_1))\cdots Z(\alpha_{|\alpha|},r(\alpha_{|\alpha|}))Z(r(x_n),x_n)\cdots Z(r(x_1),x_1),
\end{equation}
and in this case $l_S(g)\leq |\alpha|+n$.
Now, given $v\in E^0 $, we define the function $\varphi^E_v(n)=|E^{\leq n}v|$, where $E^{\leq n}v$ consists of the paths $\alpha\in E^*$ with $|\alpha|\leq n$ and $r(\alpha)=v$. Using \eqref{eq: description of fibers in graph groupoids} and \eqref{eq: product of basis elements}, one can find injective maps $E_{s(x_{j+1})}^{\leq (n-j)} \hookrightarrow B_{(\G_{E})_x}(n)$, for any $0 \leq j \leq n$, and $B_{(\G_{E})_x}(n) \hookrightarrow \bigsqcup_{i = 0}^{n} E_{s(x_{i+1})}^{\leq (n-i)}$, where the latter is a union of disjoint copies of the sets $E_{s(x_{i+1})}^{\leq (n-i)}$. It follows that
\begin{equation*}
	\varphi^E_{s(x_{j+1})} (n-j) \leq |B_{(\G_E)_x}(n)| \leq \sum_{i=0}^n\varphi^E_{s(x_{i+1})}(n-i) ,
\end{equation*}
for any $0 \leq j \leq n $.

A \emph{cycle with center $v$} is a path $\alpha$ of length at least $1$ such that $s(\alpha)=r(\alpha)=v$. A cycle $\alpha$ with center $v$ is called \emph{simple} if $r(\alpha_i)\neq v$, for every $1 \leq i<|\alpha|$. Let $\alpha$ and $\beta$ be two different simple cycles with center $v$. Then there exist two different cycles $\bar{\alpha}$ and $\bar{\beta}$ with  center  $v$ and with length $K:=|\alpha||\beta|$.  Then observe that $\varphi^E_v(nK)\geq 2^n$, and thus with $x = \bar{\alpha}^{\infty}$, we have that $| B_{(\G_E)_{x}}(n K) | \geq 2^{n}$, so $\G_E$ does not have polynomial growth.

\begin{lemma}\label{lem: graphs with only simple cycles and no two different cycles have simple cycles have polynomial growth}
	Let $E$ be a finite directed graph such that every vertex has at most one simple cycle, and that each such simple cycle has length $1$. Then $\G_E$ has polynomial growth with respect to its canonical length function.
	\begin{proof}
		Let $E$ be such a finite directed graph. Let $F$ be the graph such that $E^0 = F^0$, and $F^1$ consists of $E^1$, but we also add a simple cycle of length $1$ to those vertices that are not already a center of a simple cycle. Then $\varphi^E_v(n)\leq \varphi^F_v(n)$, so we can assume without lost of generality that every vertex in $E$ is a center for a simple cycle of length $1$. Let $\hat{E}$ be the sub-graph of $E$ with $\hat{E}^0:=E^0$ and $\hat{E}^1:=E^1\setminus \{\alpha\in E^1: r(\alpha)=s(\alpha)\}$. Thus, $\hat{E}$ is an acyclic graph.  Given $v\in E^0$, let $\Delta_v$ be the set of all the maximal paths $\gamma \in\hat{E}^*$ such that $r(\gamma)=v$, and let $m_v$ be the maximal length of all the paths in $\Delta_v$. Then we have that $|E^nv|\leq c_v (n+1)^{m_v + 3}$, where $c_v=|\Delta_v|$, and hence $\varphi_v^E(n)=|E^{\leq n}v|\leq c_v (n+1)^{m_v + 4}$. Thus, if we define $c :=\max\{c_v: v\in E^0\}$ and $m:=\max\{m_v: v\in E^0\}$, we have that $\varphi_v^E(n)\leq c(n+1)^{m + 4} $, for every $n\in \N$ and $v\in E^0$, and so
		\begin{equation*}
			|B_{(\G_E)_x}(n)| \leq \sum_{i=0}^n\varphi^E_{s(x_{i+1})}(n-i)\leq c (n+1)^{m + 5} ,
		\end{equation*}
		for every $x\in \G_E^{(0)}$. Therefore, $\G_E$ has polynomial growth. 
	\end{proof}
\end{lemma}

\begin{proposition}\label{Growth_Graphs}
	Let $\G_E$ be the graph groupoid associated to a finite graph $E$. Then $\G_E$ has polynomial growth with respect to its canonical length function if and only if every vertex of $E$ has at most one simple cycle. 
\end{proposition}
\begin{proof}
	We have seen that if there exists a vertex $v$ with at least two different simple cycles, then $|B_{(\G_E)_x}(n)|$ has exponential growth for some $x\in \G_E^{(0)}$. Now, let us suppose that every vertex of $E$ has at most one simple cycle. Let $U$ be a maximal subset of vertices of $E^0$ such that two vertices of $U$ do not belong to the same cycle. Put $X:=\bigcup_{v\in U} Z(v)$. Then $(\G_E)|_X$ is isomorphic to $\G_F $, where $F$ is a finite graph such that every vertex of $F$ has at most one simple cycle and every simple cycle has length $1$ (see \cite[Proposition 5.4 (3)]{GraphClassification}). By Lemma \ref{lem: graphs with only simple cycles and no two different cycles have simple cycles have polynomial growth}, $\G_F$ has polynomial growth. Since $\G_E$ is Kakutani equivalent to $\G_F $, and both are ample groupoids with $\sigma$-compact unit spaces, it follows by Lemma \ref{lem: polynomial growth stable under kakutani equivalence} that $\G_E$ has polynomial growth.  
\end{proof}

\begin{example} \label{example: graph giving Lp-toeplitz}
	Let $E$ be the following graph:
	
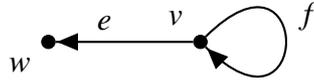
\[\begin{tikzpicture}[vertex/.style={circle, draw = black, fill = black, inner sep=0pt,minimum size=5pt}, implies/.style={double,double equal sign distance,-implies}]
		\node[vertex] (v) at (2,0) [label=above left:$v$] {};
		\node[vertex] (w) at (0,0) [label=below left:$w$] {};
		
		\path (v) edge[thick, decoration={markings, mark=at position 0.99 with {\arrow{triangle 45}}}, postaction={decorate} ] node[above left] {$e$} (w)
		edge[thick, loop, min distance = 20mm, looseness = 10, out = 45, in = 315, decoration={markings, mark=at position 0.99 with {\arrow{triangle 45}}}, postaction={decorate}]  node[above right] {$f$} (v)  ; 
	\end{tikzpicture}\]
	and let $\G_E$ the associated graph groupoid, which by Proposition \ref{Growth_Graphs} has polynomial growth. Let $p\in (1,\infty)$ and let $S\in B(\ell^p(\N))$ be the operator such that $S(e_i)=e_{i+1}$ for every $i\in \N$. Let $\mathcal{T}_p$ be the $L^p$-operator algebra generated by $S$ and its revert operator $T$. We claim that $F^p_\lambda(\G_E)$ is isometrically isomorphic to $\mathcal{T}_p$. Indeed, let $a=\chi_{Z(e,w)}+\chi_{Z(f,v)}$ and $b=\chi_{Z(w,e)}+\chi_{Z(v,f)}$ be two functions in $C_c(\G_E)$. It is straightforward to check that  $a$ and $b$ generate $C_c(\G_E)$. Now, the map $\Lambda: C_c(\G_E)\to \mathcal{T}_p$ given by $\Lambda(a)=S$ and $\Lambda(b)=T$ is an algebra homomorphism. Let $\lambda_x \colon C_c(\G_E)\to B(\ell^p((\G_E)_x))$ be the left regular representation restricted to $x\in \G_E^{(0)}$, and let $x=f^\infty\in \G_E^{(0)}$. Observe that $(\G_E)_x=\{(f^{\infty},n,f^{\infty}):n\in\Z\}$ and that   $\lambda_x(a)(e_{(f^{\infty},n,f^{\infty})})=e_{(f^{\infty},n+1,f^{\infty})}$  and $\lambda_x(b)(e_{(f^{\infty},n,f^{\infty})})=e_{(f^{\infty},n-1,f^{\infty})}$. Now let $x\in \G_E^{(0)}\setminus \{f^\infty\}=\{w,e,fe,f^2e,\ldots\}$. Then we have that $(\G_E)_x=\{(\alpha,|\alpha|-|x|,x): \alpha\in E^* \text{ such that }r(\alpha)=w \}$. Identifying $(\G_E)_x$ with $\N$ via the map $(\alpha,|\alpha|-|x|,x)\mapsto |\alpha|+1$, $\ell^p((\G_E)_x)$ is isometrically isomorphic to $\ell^p(\N)$, and with this identification we have that $\lambda_x(a)=S$ and $\lambda_x(b)=T$. But then given $\xi\in C_c(\G_E)$, we have that  
	\begin{equation*}
		\lVert\xi\rVert_{F^p_\lambda(\G_E)}=\sup\{\lVert\lambda_x(\xi)\rVert : x\in \G_E^{(0)} \}=\max\{ \lVert\lambda_{f^\infty}(\xi)\rVert, \lVert\lambda_{w}(\xi)\rVert \}.
	\end{equation*}
	But $\lVert\lambda_{w}(\xi)\rVert=\lVert\Lambda(\xi)\rVert_{\mathcal{T}_p}$. Now by \cite[Theorem 3.7]{WangWang}, we have that $\lVert\lambda_{f^\infty}(\xi)\rVert=\inf \{\lVert\Lambda(\xi)-R \rVert_{\mathcal{T}_p}: R\in K(\ell^p(\N))\}$, but then $\lVert\lambda_{f^\infty}(\xi)\rVert\leq \lVert\Lambda(\xi)\rVert_{\mathcal{T}_p}$. Therefore, $\lVert\xi\rVert_{F^p_\lambda(\G_E)}=\lVert\Lambda(\xi)\rVert_{\mathcal{T}_p}$. Thus, $\Lambda$ extends to an isometric isomorphism $\Lambda:F^p_\lambda(\G_E)\to \mathcal{T}_p$, as desired. 
	
	By Theorem \ref{thm: Polynomial growth implies all K-groups are isomorphic}, we have that the groups $K_*(\mathcal{T}_p)$ are independent of $p \in (1, \infty)$, for $\ast = 0,1$. Since we know already that $K_{0}(\mathcal{T}_2) \cong \Z$ and $K_1 (\mathcal{T}_2) = 0 $, we must have $K_{0} (\mathcal{T}_p) \cong \Z$ and $K_1 (\mathcal{T}_p) = 0 $, for all $p \in (1, \infty)$. Therefore, the above provides an alternative way to prove \cite[Theorem 4.3]{WangWang}.
\end{example}

\appendix 
\setcounter{section}{1}
\setcounter{lemma}{0}

\section*{Appendix A. The twisted case} \label{appendix:twisted-case}

In this section, we introduce reduced twisted groupoid $L^p$-operator algebras associated to normalized continuous $2$-cocycles and briefly outline how the main results of this article carry over to this more general setting. Note however that a $2$-cocycle twist is not the most general notion of a twist over a groupoid; see for example \ \cite[Chapter 11.1]{SimsNotesOnGroupoids} for details. 

\begin{definition}\label{def: normalized 2 cocycle}
	A normalized continuous $2$-cocycle on an étale groupoid $\G$ is a continuous map $\sigma \colon \G^{(2)} \to \T$ satisfying 
	\begin{itemize}
		\item[(1)] $\sigma(r(x) , x) = \sigma(x, s(x)) = 1$, for all $x \in \G$; 
		\item[(2)] $\sigma(x,y) \sigma(xy, z) = \sigma(x,yz) \sigma (y,z)$, for all $(x,y),(y,z) \in \G^{(2)}$.
	\end{itemize}
\end{definition}

\begin{definition}\label{def: cohomologous cocycles}
	Two normalized continuous $2$-cocycles $\omega$ and $\sigma$ are said to be cohomologous if there is a continuous map $\gamma \colon \G \to \T$, such that for all $(x,y) \in \G^{(2)}$, we have
	\begin{equation*}
		\sigma(x,y) \overline{\omega(x,y)} = \gamma(x) \gamma(y) \overline{\gamma(xy)} .
	\end{equation*}
\end{definition}

Given a normalized continuous $2$-cocycle $\sigma$ on an étale groupoid $\G$, we define the $\sigma$-twisted convolution algebra $C_c (\G, \sigma)$ as the associative $\C$-algebra which as a set is just $C_c (\G)$, but with $\sigma$-twisted multiplication given by 
\begin{equation*}
	f \ast_{\sigma} g (x) = \sum_{y \in \G_{s(x)}} f(xy^{-1}) g(y) \sigma(xy^{-1}, y) = \sum_{y \in \G^{r(x)}} f(y) g(y^{-1}x) \sigma(y, y^{-1}x) .
\end{equation*}
Under the $I$-norm, $C_c (\G, \sigma)$ is a normed algebra. Given $u \in \G^{(0)}$, the $\sigma$-twisted left regular representation at $u$ is denoted $\lambda_{u}^{\sigma}$ and is the contractive algebra homomorphism $\lambda_{u}^{\sigma} \colon C_c (\G, \sigma) \to B(\ell^p (\G_u))$, given by 
\begin{equation*}
	\lambda_{u}^{\sigma}(f) (\xi) (x) = \sum_{y \in \G_{u}} f(xy^{-1}) \xi(y) \sigma(x y^{-1}, y) ,
\end{equation*}
for $\xi \in \ell^p (\G_u)$ and $x \in \G_u$. The reduced $\sigma$-twisted groupoid $L^p$-operator algebra associated to $\G$ is denoted $\Fp (\G, \sigma)$ and is by definition the closure of $C_c(\G, \sigma)$ under the norm 
\begin{equation*}
	\lVert f \rVert_{\Fp(\G, \sigma)} := \sup_{u \in \G^{(0)}} \lVert \lambda_{u}^{\sigma} (f) \rVert_{B(\ell^p (\G_u))}.
\end{equation*}
Since $\bigoplus_{u \in \G^{(0)}} \lambda_{u}^{\sigma}$ is an isometric representation of $\Fp (\G, \sigma)$ on the $L^p$-space $\bigoplus_{u \in \G^{(0)}} \ell^{p}(\G_u)$, $\Fp(\G, \sigma)$ is an $L^p$-operator algebra. It is unital if and only if $\G^{(0)}$ is compact, in which case the indicator function of the unit space is the identity element. When $\sigma$ is the trivial $2$-cocycle, that is, the $2$-cocycle such that $\sigma(x,y) = 1$ for all $(x,y) \in \G^{(2)}$, twisted convolution is just the usual convolution, and $\Fp (\G, \sigma) = \Fp (\G)$. 

It is straightforward to show that $ \lVert f \rVert_{\infty} \leq \lVert f \rVert_{\Fp(\G,\sigma)}$, for all $f \in C_c (\G, \sigma)$, and so the identity on $C_c (\G, \sigma)$ extends to a contractive linear map $j_{p}^{\sigma} \colon \Fp (\G, \sigma) \to C_0 (\G)$. In fact, the map $j_{p}^{\sigma}$, like in the case of the trivial twist, takes the form 
\begin{equation*}
	j_{p}^{\sigma} (a)(x) = \langle \lambda_{s(x)}^{\sigma} (a) (\delta_{s(x)}) , \delta_{x} \rangle,
\end{equation*}
for any $a \in \Fp (\G, \sigma)$ and $x \in \G$, where $\langle \cdot, \cdot \rangle$ is the usual duality product defined in Section \ref{sec: preliminaries}. The statements in \cite[Proposition 4.7 and Proposition 4.9]{CGTRigidityResultsForLpOp} generalizes readily to give the analogous ones in the twisted case; that is, the map $j_{p}^{\sigma}$ is injective, and moreover 
\begin{equation*}
	j_{p}^{\sigma} (ab) (x) = j_{p}^{\sigma} (a) \ast_{\sigma} j_{p}^{\sigma} (b) (x) ,
\end{equation*}
for all $a,b \in \Fp(\G, \sigma)$ and $x \in \G$. Also, Lemma \ref{lem: involution extends to isometric anti-homomorphism} generalizes to the statement that the twisted involution $^\ast$, given by $f^{\ast} (x) = \overline{f(x^{-1})} \overline{\sigma(x^{-1},x)}$ for $f \in C_c (\G)$, extends to an isometric anti-isomorphism $^\ast \colon \Fp(\G, \sigma) \to F_{\lambda}^{q}(\G, \sigma)$, and $j_{q}^{\sigma}(a^{\ast}) = j_{p}^{\sigma}(a)^{\ast}$, for any $a \in \Fp(\G, \sigma)$.
\begin{definition}\label{def: twisted RD def}
	Let $\sigma$ be a normalized continuous $2$-cocycle on an étale groupoid $\G$, let $l$ be a locally bounded length function on $\G$, and let $p \in [1, \infty)$. We say $\G$ has \emph{$\sigma$-twisted property $RD_p$ with respect to $l$} if there exist $c > 0$ and $k \in \N_0$ such that 
	\begin{equation*}
		\lVert f \rVert_{\Fp(\G, \sigma)} \leq c \lVert f \rVert_{p,k} ,
	\end{equation*}
	for all $f \in C_c (\G, \sigma)$. We say $\G$ has \emph{$\sigma$-twisted property $RD_p$} if it has $\sigma$-twisted property $RD_p$ with respect to some locally bounded length function.
\end{definition}
It is straightforward to show the inequality $\lVert f \rVert_{\Fp(\G, \sigma)} \leq \lVert |f| \rVert_{\Fp (\G)}$, so that property $RD_p$ implies $\sigma$-twisted property $RD_p$ for any normalized continuous $2$-cocycle $\sigma$. Also, it follows readily from the definitions that if $\omega$ and $\sigma$ are cohomologous $2$-cocycles, then $\G$ has $\sigma$-twisted property $RD_p$ if and only if $\G$ has $\omega$-twisted property $RD_p$. Thus, if $\sigma$ is a continuous $2$-cocycle that is cohomologous to the trivial one, then $\sigma$-twisted property $RD_p$ implies $\omega$-twisted property $RD_p$ for all continuous $2$-cocycles $\omega$. It is not known to the authors whether $\sigma$-twisted property $RD_p$, for a continuous $2$-cocycle $\sigma$ that is not in the cohomology class of the trivial one, implies property $RD_p$.

Supposing an étale groupoid $\G$ is equipped with a locally bounded length function $l$ and $\sigma$ is a continuous $2$-cocycle, one can show in a completely analogous manner as in \cite[Lemma 3.3]{hou2017spectral} that $\G$ has $\sigma$-twisted property $RD_p$ if and only if $S_{p}^{l}(\G)$ is continuously included in $\Fp(\G, \sigma)$ via ${j_{p}^{\sigma}}^{-1}$. Moreover, endowing $S_{p}^{l}(\G)$ with $\sigma$-twisted convolution, $S_{p}^{l}(\G)$ becomes a Fréchet algebra if $\G$ has property $RD_p$, and, like in Proposition \ref{prop: Space of rapidly decreasing functions form a Fréchet algebra}, it may be identified as a Fréchet subalgebra of $\Fp (\G, \sigma)$ and $F_{\lambda}^{q}(\G, \sigma)$ via ${j_{p}^{\sigma}}^{-1}$ and ${j_{q}^{\sigma}}^{-1}$ respectively. It is not clear to the authors that $S_{p}^{l}(\G)$ is a Fréchet algebra under twisted convolution if $\G$ has \emph{twisted property $RD_p$}; the proof in Proposition \ref{prop: Space of rapidly decreasing functions form a Fréchet algebra} does not work in the twisted case, because it relies on an inequality of positive functions under the usual convolution product.
In Lemma \ref{lem: translation set is an algebra} the algebra $\C^p [\G_u]$ clearly contains $\lambda_{u}^{\sigma} (C_c (\G , \sigma) )$, and Lemma \ref{lem: dense spectral invariant Frechet algebra, Ji} and Lemma \ref{lem: derivation at u is closed} remain the same in the twisted case. The arguments in Theorem \ref{thm: RD_q and RD_p implies isomorphisms in K-theory} and Theorem \ref{thm: Polynomial growth implies all K-groups are isomorphic} extends in the natural manner to give the next two results. By $S_{p}^{l}(\G, \sigma)$ we mean $S_{p}^{l}(\G)$ endowed with the twisted convolution product.

\begin{proposition}\label{prop: twisted version of RDp and RDq imply isomorphism in K-theory}
	Let $\G$ be an étale groupoid which has property $RD_p$ and $RD_q$, for $p,q \in (1, \infty)$ Hölder conjugate, with respect to some continuous length function $l$, and let $\sigma$ be a continuous $2$-cocycle on $\G$. We may identify $S_{p}^{l}(\G, \sigma)$ as a Fréchet subalgebra of $\Fp(\G , \sigma)$ and $F_{\lambda}^{q}(\G, \sigma)$. Under these identifications, $S_{p}^{l}(\G, \sigma)$ is a spectral invariant Fréchet subalgebra of $F_{\lambda}^{p}(\G, \sigma) $ and $F_{\lambda}^{q}(\G, \sigma)$, and the inclusions induce isomorphisms on their $K$-theories; in particular, 
	\begin{equation*}
		K_{\ast} (F_{\lambda}^{p}(\G, \sigma)) \cong K_{\ast} (S_{p}^{l}(\G, \sigma))  \cong K_{\ast}(F_{\lambda}^{q}(\G, \sigma)).
	\end{equation*}
\end{proposition}

\begin{proposition} \label{prop: twisted version of polynomial growth imply all K-groups are isomorphic}
	Let $\G$ be an étale groupoid and let $\sigma$ be a continuous $2$-cocycle on $\G$. If $\G$ has polynomial growth with respect to a continuous length function, then the groups $ K_{\ast}(F_{\lambda}^{p}(\G, \sigma))$ are independent of $p \in (1,\infty)$.
\end{proposition}

\begin{example}{($L^p$-Noncommutative Torus)}
	Fix a real skew-symmetric $n\times n$ matrix $\Theta$, and set 
	\begin{equation*}
		\begin{split}
			\sigma \colon \Z^n \times \Z^n &\to \T \\
			(v,w) &\mapsto e^{\pi i v \cdot \Theta w}
		\end{split}
	\end{equation*} 
	for all $v,w \in \Z^n$. Then $\sigma$ is a (continuous) $2$-cocycle on the group $\Z^n$. Due to amenability of $\Z^n$, the noncommutative $n$-torus determined by $\Theta$ can be realized as $C^*_r (\Z^n , \sigma)$. We may then define the $L^p$-noncommutative $n$-torus determined by $\Theta$ as $\Fp (\Z^n, \sigma)$. It is well-known that $K_0 (C^*_r (\Z^n , \sigma)) \cong \Z^{2^{n-1}} \cong K_1 (C^*_r (\Z^n , \sigma))$, see \cite{RieffelCaseStudy}. Since $\Z^n$ has polynomial growth, it follows from Proposition \ref{prop: twisted version of polynomial growth imply all K-groups are isomorphic} that
	\begin{equation*}
		\begin{split}
			K_{0} (F_{\lambda}^{p}(\Z^n, \sigma)) \cong K_{0}(C^*_r (\Z^n, \sigma)) \cong \Z^{2^{n-1}}
		\end{split}
	\end{equation*}
	and 
	\begin{equation*}
		\begin{split}
			K_{1} (F_{\lambda}^{p}(\Z^n, \sigma)) \cong K_{1}(C^*_r (\Z^n, \sigma)) \cong \Z^{2^{n-1}}
		\end{split}
	\end{equation*}
	for all $p \in (1, \infty)$. Note that we are not claiming that the isomorphism for $K_0$ is order-preserving.
\end{example}

\printbibliography

\end{document}